\documentclass[12pt,oneside,reqno]{amsart}        

\usepackage{etex}
\reserveinserts{28}  

\usepackage{amssymb,amsthm,amsmath,mathrsfs,comment}
\usepackage{amstext,amsxtra}  
\usepackage{fullpage,colonequals}
\usepackage{bm}       
\usepackage{mathtools} 
\usepackage[all]{xy}
\usepackage{enumitem}

\usepackage{caption} 
\usepackage[usenames,dvipsnames,svgnames,table]{xcolor}                                                        

\usepackage{pictex}
\usepackage{multirow}
\usepackage{multicol}

\usepackage[pdfpagelabels]{hyperref}
\hypersetup{pdftitle={},pdfauthor={}} 
\hypersetup{colorlinks=true,linkcolor=blue,anchorcolor=blue,citecolor=blue}

\theoremstyle{plain}

\newtheorem*{question*}{Question}

\newtheorem{theorem}{Theorem}

\newtheorem{proposition}{Proposition}
\newtheorem{lemma}{Lemma}
\newtheorem{corollary}{Corollary}

\newtheorem*{definition*}{Definition}
\newtheorem*{example*}{Example}
\newtheorem*{theorem*}{Theorem} 
\newtheorem*{proposition*}{Proposition}

\theoremstyle{definition}

\theoremstyle{remark}
\newtheorem{remark}{Remark}
\newtheorem*{remark*}{Remark}

\newtheorem{example}{Example}

\newcommand{\FF}{\mathbb F}

\newcommand{\calO}{\mathcal{O}}

\newenvironment{enumalph}
{\begin{enumerate}}
{\end{enumerate}}

\def\ZZ{\mathbb Z}
\def\QQ{\mathbb Q}
\def\RR{\mathbb R}
\def\CC{\mathbb C}

\def\OO{\calO}   

\def\epsilon{\varepsilon}
\def\phi{\varphi}

\def\mod{\,\text{mod}\,}

\def\EK4{E_{K,4} }

\def\norm#1{\text{Norm}_{k/\QQ} ( #1 )}

\def\nnorm#1#2{\text{Norm}_{#1/\QQ} ( #2 )}              
\def\nnnorm#1#2#3{\text{Norm}_{#1/#2} ( #3 )}

\def\gp#1{\langle \, #1 \, \rangle}    

\DeclareMathOperator{\iso}{\simeq}

\DeclareMathOperator{\Gal}{Gal}

\DeclareMathOperator{\sign}{sign}

\begin{document}

\title[Unit Signatures in Real Biquadratic and Multiquadratic Number Fields]
{Unit Signatures in Real Biquadratic and Multiquadratic Number Fields}

\author{David S.\ Dummit}
\address{Department of Mathematics, University of Vermont, Lord House, 16 Colchester Ave., Burlington, VT 05405, USA}
\email{dummit@math.uvm.edu}

\author{Hershy Kisilevsky}
\address{Hershy Kisilevsky, Department of Mathematics and Statistics and CICMA, Concordia University, 1455 de Maisonneuve  Blvd. West, Montr\'eal, Quebec, H3G 1M8, CANADA}
\email{hershy.kisilevsky@concordia.ca}

\date{\today}

\begin{abstract}
We consider the signature rank of the units in real multiquadratic fields.  When
the three quadratic subfields of a real biquadratic field $K$ either 
(a) all have
signature rank 2 (that is, fundamental units of norm $-1$),
or (b) all have
signature rank 1 (that is, have totally positive fundamental units), 
we provide
explicit examples to show there exist infinitely many $K$ having each of the possible
unit signature ranks (namely signature rank 3 or 4 in case (a) and signature
rank 1,2, or 3 in case (b)).  We make some additional remarks for higher rank
real multiquadratic fields, in particular proving the rank of the totally positive
units modulo squares in such extensions (hence also in the totally real subfield
of cyclotomic fields) can be arbitrarily large.
\end{abstract}



\maketitle

\section{Introduction}

Suppose $F$ is a totally real field of degree $n$ over $\QQ$ and $0 \ne \alpha \in F$.  For any of the $n$
real places $v:F \hookrightarrow \RR$ of $F$, let $\sign(v(\alpha))$  denote the sign of
the element $v(\alpha) \in \RR$.  
We frequently view the sign of an element as lying in the additive group $\FF_2$ rather than the
multiplicative group $\{ \pm 1 \}$ (so having value 0 if $v(\alpha) > 0$ and value 1 if $v(\alpha) < 0$),
and the point of view being used should be clear from the context.
The $n$-tuple of signs $(\dots , \sign(v(\alpha)), \dots)$ is called the {\it signature} of $\alpha$.  

When $F/\QQ$ is Galois and one real embedding of $F$ is fixed, we can view $F$ as
a subfield of $\RR$ and the real embeddings of $F$ are indexed by the elements $\sigma$ in
$\Gal(F/\QQ)$.  In this case, the signature of $\alpha$ is given by the $n$-tuple of signs of
the real numbers $\sigma(\alpha)$.

The collection of signatures of the units of $F$ (viewed additively) is a subspace of $\FF_2^n$ called the 
{\it unit signature group} of $F$ and the rank of this subspace is called the {\it (unit) signature rank} of $F$.
As in \cite{DDK}, define the {\it (unit signature rank) ``deficiency'' of $F$}, denoted $\delta(F)$, to be 
the corank of the unit signature group of $F$, that is, 
$n$ minus the signature rank of the units of $F$.  The deficiency of $F$ is  
the nonnegative difference between the
unit signature rank of $F$ and its maximum possible value and is 0 if and only if there
are units of every possible signature type.  The deficiency is also the rank of the
group of totally positive units of $F$ modulo squares, and the class number of
$F$ times $2^{\delta(F)}$ gives the strict (or narrow) class number of $F$.  
It is trivial that the unit signature rank never decreases in a totally real extension of a totally real field.
As noted in \cite{DDK}, by a result of Edgar, Mollin and Peterson (\cite[Theorem 2.1]{EMP}), the unit signature
rank deficiency
also never decreases in a totally real extension of a totally real field.

In this paper we first consider the unit signature rank of real biquadratic fields $K$.
By a result of Kuroda (\cite[Satz 11]{Kur}), which we summarize in the following proposition,
the unit group, $E_K$, of a real biquadratic field $K$ 
can be obtained by extracting appropriate 
square roots of elements in the group generated by the units from the quadratic
subfields of $K$.   

\begin{proposition} \label{prop:unitstructure}
Suppose $K$ is a real biquadratic extension of $\QQ$ having quadratic subfields $k_1$, $k_2$, $k_3$, with corresponding
fundamental units $\epsilon_1$, $\epsilon_2$ and $\epsilon_3$. 

\begin{enumalph}

\item
The group $E_K / \gp{-1, \epsilon_1 , \epsilon_2, \epsilon_3}$ is an elementary
abelian 2-group (of rank at most 3).

\item 
If not all of $\nnorm{k_1}{\epsilon_1}$, $\nnorm{k_2}{\epsilon_2}$, and $\nnorm{k_3}{\epsilon_3}$
are equal to $-1$, then there are (up to a permutation of the quadratic subfields) precisely 7 possibilities 
for the unit group $E_K$:
\vspace{-10pt}
\begin{multicols}{2} 
\begin{enumerate}

\item[{1.}]  
$ \gp{-1, \epsilon_1 , \epsilon_2, \epsilon_3} $

\item[{2.}] 
$ \gp{-1, \sqrt{\epsilon_1}, \epsilon_2, \epsilon_3} $

\item[{3.}] 
$ \gp{-1, \sqrt{\epsilon_1}, \sqrt{\epsilon_2}, \epsilon_3} $

\item[{4.}] 
$ \gp{-1, \sqrt{\epsilon_1 \epsilon_2}, \epsilon_2, \epsilon_3} $

\item[{5.}] 
$ \gp{-1, \sqrt{\epsilon_1 \epsilon_2}, \sqrt{\epsilon_3}, \epsilon_2} $

\item[{6.}] 
$ \gp{-1, \sqrt{\epsilon_1 \epsilon_2}, \sqrt{\epsilon_2 \epsilon_3}, \sqrt{\epsilon_3 \epsilon_1} } $

\item[{7.}] 
$ \gp{-1, \epsilon_1 , \epsilon_2, \sqrt{ \epsilon_1 \epsilon_2 \epsilon_3} } $

\end{enumerate}
\end{multicols}
\vskip -10pt
\noindent
where in each case any unit appearing in a square root has norm $+1$.

\smallskip

\item

If $\nnorm{k_1}{\epsilon_1}= \nnorm{k_2}{\epsilon_2} = \nnorm{k_3}{\epsilon_3} = -1$, there are
precisely 2 possibilities for the unit group $E_K$:

\begin{enumerate}

\item[{1.}] 
$ \gp{-1, \epsilon_1 , \epsilon_2, \epsilon_3} $

\item[{2.}] 
$ \gp{-1, \epsilon_1 , \epsilon_2, \sqrt{ \epsilon_1 \epsilon_2 \epsilon_3} } $

\end{enumerate}

\end{enumalph}

\end{proposition}

In \cite{Kub}, Kubota proved that each of the possibilites in Proposition \ref{prop:unitstructure}
occurs infinitely often, providing an explicit infinite family for each case.

Proposition \ref{prop:unitstructure} shows that the unit signature rank of the biquadratic field
$K$ is influenced both by the signatures of the fundamental units of the three quadratic 
subfields of $K$ and by how those fundamental units are situated in the units of $K$.
In this paper we concentrate on the two extreme cases: (a) where 
the three quadratic subfields either all have deficiency 0 (that is, fundamental units of norm $-1$),
and (b) where the three quadratic subfields all have deficiency 1 (that is, have totally positive fundamental units).
In case (a), the signature rank of the units of $K$ is either 3 or 4 and we provide an
explicit infinite family for each possibility 
(see Theorems \ref{thm:def0rank4} and \ref{thm:def0rank3} in Section~\ref{sec:deficiency0}).
In case (b), the signature rank of the units of $K$ is either 1, 2, or 3 and again we provide an
explicit infinite family for each possibility 
(see Theorems \ref{thm:def1ranks2and3} and \ref{theorem:deficiency3} in Section~\ref{sec:deficiency1}).

In the final section we prove some results for higher rank multiquadratic fields.  We prove that the 
unit signature rank deficiency can be unbounded in multiquadratic extensions, exhibiting 
specific families where the number of totally positive units that are independent modulo squares tends to infinity
(see Theorems \ref{theorem:qmultifields} and \ref{theorem:unboundeddeficiency}).  In particular
this proves the unit signature rank deficiency can be arbitrarily large in real cyclotomic
fields (see Theorem \ref{thm:cyclotomicrank}).

\section{Preliminaries on Units of norm $+1$ in real quadratic fields}

Suppose $k = \QQ(\sqrt d)$ is a real quadratic field ($d> 1$ a squarefee integer) with fundamental unit $\epsilon$, normalized
as usual so that $\epsilon > 1$ with respect to the embedding of $k$ into $\RR$ for which $\sqrt d > 0$.

If $\text{Norm}_{k/\QQ} (\epsilon) = +1$ then, by Hilbert's Theorem 90, 
%
\begin{equation} \label{eq:alphadef}
\epsilon =\sigma (\alpha) / \alpha 
\end{equation}
for some $\alpha \in \QQ(\sqrt d)$, which may be assumed to be an algebraic integer (for example, take
$\alpha = \sigma (\epsilon) + 1$).  
The principal ideal $(\alpha)$ is invariant under $\sigma$ (that is, is an ambiguous ideal)
since $\sigma (\alpha)$ differs from $\alpha$ by a unit, so the element $\alpha$ satisfying \eqref{eq:alphadef} 
can be further chosen so that the ideal $(\alpha)$ is the product of
distinct ramified primes.  Also, 
\eqref{eq:alphadef} shows that $\alpha$ and $\sigma (\alpha)$ have the same sign (in either embedding of $k$ into $\RR$), so
multiplying $\alpha$ by $-1$, if necessary, we may also assume that $\alpha$ is totally positive.  Then 
$0 < \alpha < \sigma(\alpha)$ in the embedding for which $\sqrt d > 0$ since $\epsilon > 1$ in this
embedding.  With these additional requirements, the element $\alpha$ is unique.  Let $m$ denote the norm of $\alpha$:
\begin{equation} \label{eq:mdef}
m = \alpha \ \sigma (\alpha) ,
\end{equation}
so that $m$ is a positive squarefree integer dividing the discriminant of $k$. 

From \eqref{eq:alphadef} we have
\begin{equation} \label{eq:meps}
m \ \epsilon = \alpha \, \sigma (\alpha) \dfrac{ \sigma (\alpha) }{\alpha} = \sigma (\alpha)^2
\end{equation}
so 
$m \epsilon$ is a square in $k^*$.  We have
\begin{equation} \label{eq:epsplus1}
\epsilon + 1 = \dfrac{ \sigma (\alpha) }{\alpha} + 1 = \dfrac{ \sigma (\alpha) + \alpha }{\alpha},
\end{equation}
so
\begin{equation} \label{eq:Normepsplus1}
\norm{\epsilon + 1} = \dfrac{ (\sigma (\alpha) + \alpha )^2 }{m},
\end{equation}
where $\alpha + \sigma (\alpha) \in \ZZ$.  Hence $m \, \norm{\epsilon + 1}$ is a square in $\ZZ$.  It follows that
$m$ is the squarefree part of $\norm{\epsilon + 1}$  and that $m$ is positive since $\norm{\epsilon + 1} > 1$
(note $\epsilon + 1 > 1$ with respect to both embeddings of $k$).

Finally, since $m \epsilon$ is a square in $k$, $m$ cannot equal 1 or $d$, as both of these values are
squares in $k$ and $\epsilon$ is not a square.

We summarize this in the following proposition.

\begin{proposition}[\cite{Kub}] \label{prop:m}
Suppose $\norm{\epsilon} = +1$ in $k = \QQ(\sqrt d)$ as above, and let
$m$ denote the squarefree part of the positive integer $\norm{\epsilon + 1}$. Then $m > 1$, 
$m$ divides the discriminant of $k$, $m \ne d$, $m$ is the norm of an integer in $k$, and 
$m \epsilon$ is a square in $k$.
\end{proposition}

From \eqref{eq:alphadef} we have
\begin{equation} \label{eq:epsplus1}
\epsilon - 1 = \dfrac{ \sigma (\alpha) }{\alpha} - 1 = \dfrac{ \sigma (\alpha) - \alpha }{\alpha},
\end{equation}
so
\begin{equation} \label{eq:Normepsminus1}
\norm{\epsilon - 1} = \dfrac{ -(\sigma (\alpha)- \alpha)^2 }{m}.
\end{equation}

\begin{remark}
We have $(\sigma (\alpha)- \alpha)^2 = D \cdot [ \OO_K : \ZZ[\alpha] ]^2$ where $D$ is the discriminant of $k$ 
(note the right hand side is positive, as is the left hand side). Then
\begin{equation} 
-\norm{\epsilon - 1} = \dfrac{D}{m} \cdot [ \OO_K : \ZZ[\alpha] ]^2 ,
\end{equation}
which proves that the squarefree part of $D/m$ is the squarefree part of $-\norm{\epsilon - 1}$.
\end{remark}

By \eqref{eq:Normepsplus1}, $\norm{\epsilon + 1}$ is a positive integer whose positive square root is given by 
\begin{equation} \label{eq:sqrtNplus1}
\sqrt{ \norm{\epsilon + 1} } = \dfrac{\sigma (\alpha) + \alpha }{\sqrt m} .
\end{equation}

Similarly, by \eqref{eq:Normepsminus1},  $-\norm{\epsilon - 1}$ is a positive integer whose positive square root is given by 
\begin{equation}  \label{eq:sqrtNminus1}
\sqrt{ - \norm{\epsilon - 1} } = \dfrac{\sigma (\alpha) - \alpha}{\sqrt m} 
\end{equation}
(recall that $\sigma(\alpha) > \alpha$ in the embedding where $\sqrt d > 0$).

Write the normalized element $\alpha$ satisfying \eqref{eq:alphadef} as
\begin{equation} \label{eq:alpha}
\alpha = A + B \sqrt d .
\end{equation}
Then $\sigma(\alpha) > \alpha > 0$ and the definition of $m$ gives
\begin{equation} \label{eq:ABproperties}
A > 0 \ , \  B < 0\text{ and } A^2 - d B^2 = m.
\end{equation}

By \eqref{eq:sqrtNplus1} and 
\eqref{eq:sqrtNminus1}, $A$ is $1/2$ of the positive square root of (the positive square integer) $m \, \norm{\epsilon + 1}$ and 
$B$ is $1/2$ of the negative square root of 
(the positive square integer) $- m \, \norm{\epsilon - 1}/d$.  Hence $m$, $A$, and $B$ are easily computed once
$\epsilon$ is determined (by a continued fraction algorithm, for example).

\begin{proposition} \label{prop:sqrtepsilon}
Suppose the fundamental unit $\epsilon$ in $k = \QQ(\sqrt d)$ has norm $+1$ and $\alpha$, $m$ are as above.
If $A$ and $B$ are defined by \eqref{eq:alpha}, then $ A > 0$, $B < 0$, $A^2 - d B^2 = m$ and
\begin{equation} \label{eq:sqrtepsilon}
\sqrt{\epsilon} = \frac{1}{\sqrt m} (A - B \sqrt{d})
\end{equation}
(all square roots positive). 
\end{proposition}

\begin{proof}
This follows from \eqref{eq:meps} by taking (positive) square roots.
\end{proof}

\begin{remark}
The integer $m = m_\epsilon$ in Proposition \ref{prop:m} can be defined for
any (not necessarily fundamental) unit $\epsilon$ in $k$ whose norm is $+1$, and the resulting $m_\epsilon$ satisfies the 
properties in Proposition \ref{prop:m} with 
the exception that $m_\epsilon$ could be 1 or $d$.  If the unit $\epsilon$ also satisfies $\epsilon > 1$, then
the formula for $\sqrt{\epsilon}$ in Proposition \ref{prop:sqrtepsilon} also holds.  

\end{remark}

\section{Some applications}


We give some applications of the existence of the integer $m = m_{\epsilon}$ for units $\epsilon$ of norm $+1$ in 
Proposition \ref{prop:m} (which could loosely be referred to as the `$m$-technology').  
The first is an elementary proof of a result of Dirichlet.  

\begin{proposition} \label{prop:p}
Suppose $p$ is a prime $\equiv 1$ mod 4.  If $\epsilon$ denotes the
fundamental unit of $k = \QQ(\sqrt{p})$, then $\text{Norm}_{k/\QQ} (\epsilon) = -1$.
\end{proposition}

\begin{proof}
By way of contradiction, suppose $\text{Norm}_{k/\QQ} (\epsilon) = +1$. Then the integer $m = m_\epsilon$ above
divides $p$ and is neither 1 nor $p$, which is impossible.  
\end{proof}

The next result is also due to Dirichlet and, together with the previous proposition,
provides infinitely many examples of real biquadratic fields
all of whose subfields have a fundamental unit of norm $-1$.

\begin{proposition} \label{prop:p1p2}
Suppose that $p_1$ and $p_2$ are distinct primes with $p_2 \equiv 1$ mod 4 and with $p_1 = 2$ or $p_1 \equiv 1$ mod 4,
and let $\epsilon$ be the fundamental unit of $k = \QQ(\sqrt{p_1 p_2})$.
If $(\frac{p_1}{p_2}) = -1$ then $\text{Norm}_{k/\QQ} (\epsilon) = -1$.
\end{proposition}

\begin{proof}
By way of contradiction, suppose $\text{Norm}_{k/\QQ} (\epsilon) = +1$. Then the integer $m = m_\epsilon$ above
divides $p_1 p_2$ and is neither 1 nor $p_1 p_2$, so $m = p_1$ or $p_2$.  If $ m = p_1$, then $p_1$
would be the norm of an integer from $k$, so $a^2 - p_1 p_2 b^2= 4 p_1$ (or $a^2 - 2 p_2 b^2 = 2$ if
$p_1 = 2$) would have integral solutions, which contradicts the fact
that $p_1$ is not a square mod $p_2$. If $ m = p_2$, then $a^2 - p_1 p_2 b^2= 4 p_2$ (or $a^2 - 2 p_2 b^2 = p_2$ if $p_1 = 2$)
would have integral solutions, so $p_2$ would divide $a$ and then
$p_2 (a')^2 - p_1  b^2 = 4 $ (or $p_2 (a')^2 - 2  b^2 = 1$ if $p_1 = 2$) would have integral solutions,
contradicting the fact that $-p_1$ is also not a square mod $p_2$.  
Hence $\text{Norm}_{k/\QQ} (\epsilon) = +1$
is impossible, completing the proof.
\end{proof}

The final application relates to the genus theory for real quadratic fields.  

Suppose $k = \QQ(\sqrt d)$ is a real quadratic field ($d> 1$ a squarefee integer) with fundamental unit $\epsilon$ and discriminant $D$.
If $D$ is divisible by
$t$ distinct primes then by genus theory the subgroup $C_k^+[2]$ of elements of order dividing 2 in the 
strict class group $C_k^+$ of $k$ is isomorphic to $(\ZZ/2\ZZ)^{t-1}$, i.e., the 2-rank of 
$C_k^+$ is $t-1$.  The group $C_k^+[2]$ is generated by 
the $t$ classes of the ramified primes, which satisfy a single relation.

\begin{proposition}
Suppose $k$ is a real quadratic field as above. Then the unique relation
among the classes of the ramified primes in the strict class group of $k$ is given as follows:

\begin{enumerate}

\item[(a)]
if $\text{Norm}_{k/\QQ} (\epsilon) = -1$, the product of the classes of the primes 
$\frak p$ that divide $d$ is 1, and

\item[(b)]
if $\text{Norm}_{k/\QQ} (\epsilon) = +1$, the product of the classes of the primes 
$\frak p$ that divide $m$ is 1, where $m = m_{\epsilon}$ is the positive integer associated to $\epsilon$ in 
Proposition \ref{prop:m}.

\end{enumerate}

\end{proposition}

\begin{proof}
If $\text{Norm}_{k/\QQ} (\epsilon) = -1$ then the principal ideal $(\sqrt{d})$ has $\epsilon \sqrt d$ as a 
totally positive generator, hence is trivial in the strict class group.
Since $(\sqrt{d})$ is the product of the primes $\frak p$ that divide $d$, this proves (a). 

If $\text{Norm}_{k/\QQ} (\epsilon) = +1$ and $m = m_\epsilon$, then 
$m =  \text{Norm}_{k/\QQ} (\alpha)$ as in \eqref{eq:mdef}.  
Then the principal ideal $(\alpha)$ is the product of the primes $\frak p$ that divide $m$ and 
since $\alpha$ is totally positive this product is trivial in the strict class group, which is (b).
\end{proof}

\begin{remark}
With notation as above, the 2-rank of the ordinary class group $C_k$ of $k$ is either $t-1$ or $t-2$, with the latter occuring
if and only if $d$ is divisible by a prime $q \equiv 3$ mod 4.  The principal ideal $(\sqrt d)$ is
trivial in $C_k$, so the product of the classes of the (ramified) primes $\frak p$ that divide $d$ is always
trivial.  

There are three possibilities:

\begin{enumerate}

\item
There is no element $\omega$ in $k$ with norm $-1$ (i.e., $d$ is divisible by a prime $q \equiv 3$ mod 4).  
Then $C_k^+ = (\ZZ / 2 \ZZ) \oplus C_k$ and the 2-rank of $C_k$ is $t-2$.  The subgroup $C_k[2]$ is
generated by the classes of the ramified primes, with two independent relations: the product of
the classes of the primes $\frak p$ that divide $d$ and the product of the classes of
the primes dividing $m = m_\epsilon$ are both equal to 1.

\item
There is an element $\omega$ in $k$ with norm $-1$ (i.e., $d$ is not divisible by any prime $q \equiv 3$ mod 4)
but $\epsilon$ has norm $+1$ (example: $d = 34$).  In this case $C_k$ and $C_k^+$ have the same 2-rank $t-1$ 
(but $\vert C_k^+ \vert = 2 \vert C_k \vert$).  The group $C_k[2]$ requires one generator in 
addition to the classes of the ramified primes:  since $\text{Norm}_{k/\QQ} (\omega) = -1$, the
principal ideal $(\omega)$ can be written in the form $\sigma {\frak a} / \frak a $ for some fractional
ideal $\frak a$, and the class of $\frak a$ in $C_k$ gives an element of order 2 not in the subgroup
generated by the classes of the ramified primes ($\frak a$ defines an ambiguous ideal class but is not
equivalent to an ambiguous ideal).  As in (1) there are two independent relations among the classes of
the ramified primes: the product of
the classes of the primes $\frak p$ that divide $d$ and the product of the classes of
the primes dividing $m = m_\epsilon$ are both equal to 1.

\item
If $\epsilon$ has norm $-1$, then $C_k = C_k^+$, so has 2-rank $t-1$ with $C_k[2]$ generated by the
classes of the ramified primes with unique relation that the product of the classes of the primes
dividing $d$ is equal to 1.

\end{enumerate}

\end{remark}

\section{Real biquadratic fields whose quadratic subfields all have deficiency 0}  \label{sec:deficiency0}

In this section we suppose that 
$K$ is a real biquadratic extension of $\QQ$ having quadratic subfields $k_1$, $k_2$, $k_3$, with corresponding
fundamental units $\epsilon_1$, $\epsilon_2$ and $\epsilon_3$, each of which has norm $-1$ (that is,  
the units of $k_1$, $k_2$, and $k_3$ have all possible signatures).  

In this case, the matrix of signatures (viewed additively) of $\{ -1, \epsilon_1, \epsilon_2, \epsilon_3 \}$ in $K$ is
\begin{equation} \label{eq:sigmatrix4}
\begin{pmatrix}
1 & 1 & 1 & 1 \\
0 & 1 & 0 & 1 \\
0 & 0 & 1 & 1 \\
0 & 1 & 1 & 0
\end{pmatrix},
\end{equation}
which has rank 3.  Hence the deficiency of $K$ is 0 or 1 (that is, the signature rank of the units is 4 or 3).  

By Proposition \ref{prop:unitstructure}, a set of fundamental units of $K$ 
is given either by 
$\{ \epsilon_1, \epsilon_2,\epsilon_3 \}$ or $\{ \epsilon_1,\epsilon_2 , \sqrt{ \epsilon_1 \epsilon_2 \epsilon_3 }, \}$,
depending on whether $ \epsilon_1 \epsilon_2 \epsilon_3$ is a square in $K$.  

If $\eta = \epsilon_1 \epsilon_2 \epsilon_3$ and $\sqrt \eta \in K$, then Kubota shows (\cite[Hilfsatz 3]{Kub}) that
$\nnorm{K}{\sqrt \eta} = -1$, whose proof we briefly recall since it will be useful later in the proof of 
Theorem \ref{thm:def0rank3}. 
If $\sigma$ is the nontrivial automorphism of $K$ fixing $k_1$, then $\eta^{1 + \sigma} =  \epsilon_1^2 (-1)(-1) = \epsilon_1^2$
since $\sigma$ acts non trivially on the units $\epsilon_2$ and $\epsilon_3$, each of which has norm $-1$.  Hence
$ \sqrt{\eta}^{\, \sigma} = (-1)^{\nu_1} \epsilon_1 / \sqrt{\eta}$ for some $\nu_1$.  Similarly 
$ \sqrt{\eta}^{\, \tau} = (-1)^{\nu_2} \epsilon_2 / \sqrt{\eta}$ for some $\nu_2$.  Then
\begin{equation} \label{eq:conjugateeta}
\begin{aligned}
 \sqrt{\eta}^{\, \sigma \tau} & = ((-1)^{\nu_1} \epsilon_1 / \sqrt{\eta})^\tau 
  = (-1)^{\nu_1} (- 1/\epsilon_1) (-1)^{\nu_2} \sqrt{\eta}/\epsilon_2 \\
& = (-1)^{\nu_1 + \nu_2 + 1} \sqrt{\eta} /(\epsilon_1\epsilon_2) = (-1)^{\nu_1 + \nu_2 + 1} \epsilon_3 / \sqrt{\eta}.
\end{aligned}
\end{equation} 
Taking the product of the conjugates of $\sqrt{\eta}$ gives $-1$, which shows $\nnorm{K}{\sqrt \eta} = -1$.

It follows that if $ \sqrt{\epsilon_1 \epsilon_2 \epsilon_3} \in K$ then 
an odd number of conjugates of this element are negative, so its signature is not contained in the 
group of signatures generated by $\epsilon_1, \epsilon_2, \epsilon_3$ in \eqref{eq:sigmatrix4}, so $K$ has signature rank 4.
This shows that which of the two possibilities in (c) of Proposition \ref{prop:unitstructure} occurs is determined
by the unit signature rank of $K$.  We record this in the following proposition.

\begin{proposition} \label{prop:equivalence}
Suppose $K$ is a real biquadratic field all of whose quadratic subfields have deficiency 0 (that is, 
their fundamental units $\epsilon_1$, $\epsilon_2$, and $\epsilon_3$ have norm $-1$). Then

\begin{enumerate}

\item
$ E_K = \gp{-1, \epsilon_1 , \epsilon_2, \epsilon_3} $ if and only if the unit signature rank of
$K$ is 3, or, equivalently, 

\item
$ E_K =  \gp{-1, \epsilon_1 , \epsilon_2, \sqrt{ \epsilon_1 \epsilon_2 \epsilon_3} } $ 
if and only if the unit signature rank of $K$ is 4.

\end{enumerate}

\end{proposition}

The following two theorems provide infinitely many examples of each possibility in Proposition \ref{prop:equivalence}.
We begin with unit signature rank 4.

\eject

\begin{theorem} \label{thm:def0rank4}
Suppose that $p_1$ and $p_2$ are distinct primes with $p_2 \equiv 1$ mod 4 and with $p_1 = 2$ or $p_1 \equiv 1$ mod 4.
If the fundamental unit $\epsilon$ of $k = \QQ(\sqrt{p_1 p_2})$, satisfies $\text{Norm}_{k/\QQ} (\epsilon) = -1$
(for example, if $(\frac{p_1}{p_2}) = -1$, by Proposition \ref{prop:p1p2}), then 
the unit signature rank of $K = Q(\sqrt{p_1},\sqrt{p_2})$ is 4, that is, $K$ has deficiency 0.   
A set of fundamental units for
$K$ is given by $\{ \epsilon_1, \epsilon_2, \sqrt{ \epsilon_1 \epsilon_2 \epsilon_3 } \}$
where $\epsilon_1$, $\epsilon_2$, $\epsilon_3$ are the fundamental units for the three quadratic subfields of $K$.
\end{theorem}

\begin{proof}
By genus theory, the 2-rank of the class group of $k$ is 1, that is, the 2-part of the class group of $k$ is
cyclic.  Since  $\text{Norm}_{k/\QQ} (\epsilon) = -1$, the class group and the strict class group of $k$ are
equal.  If $H_k^+$ denotes the strict 2-class field of $k$, then $H_k^+$ is a totally real cyclic extension of $k$
that contains the field $K = Q(\sqrt{p_1},\sqrt{p_2})$.  If the signature rank of $K$ were not 4, then $K$ would
have a nontrivial abelian extension $L$ that is not totally real and is unramified over $K$ at all finite primes.
Hence $L$ would be contained in the strict Hilbert 2-class field tower of $k$.  But the strict 2-class field tower of
$k$ is just $H_k^+$ (the tower terminates at the first layer because the Galois group over $k$ of a nontrivial 
second layer would be a nonabelian group with cyclic commutator quotient group).  Such an $L$ therefore 
contradicts the fact that $H_k^+$ is totally real, so the signature rank of $K$ is 4 and then the final statement follows
from (2) of Proposition \ref{prop:equivalence}.
\end{proof}

\begin{remark}
The examples of Theorem \ref{thm:def0rank4} in which $p_1$ and $p_2$ 
are both odd and 
satisfy $(\frac{p_1}{p_2}) = -1$ appear
in \cite{Kub}, where Kubota proved by a different method that they give examples where the units in 
the biquadratic field $K$ were given by case 2 of (c) in Proposition \ref{prop:unitstructure}.
This, together with Proposition \ref{prop:equivalence}, gives a slightly different proof that these
fields have unit signature rank 4, and in particular the class number
and the strict class number are equal since there are units of all possible signatures. 
\end{remark}

\begin{remark}
In fact the strict class number of the fields $K$ in Theorem \ref{thm:def0rank4}
is odd (so equal to the class number), as follows.
The extension $K/\QQ( \sqrt{p_1})$ is of degree 2, unramified outside the prime $(p_2)$, and totally ramified at
this prime, so its strict class number is odd if and only if the strict class number of $\QQ( \sqrt{p_1})$ is odd 
(see \cite[Lemma]{D1}), and $\QQ( \sqrt{p_1})$ has odd class number by genus theory.  
\end{remark}

\begin{remark}
The example $p_1 = 5$, $p_2 = 29$, with $\epsilon = 12 + \sqrt{145}$ 
(respectively, $p_1 = 2$, $p_2 = 41$, with $\epsilon = 9 + \sqrt{82}$)
shows there are
fields with $p_1 \equiv p_2 \equiv 1$ mod 4 
(respectively, $p_1 = 2$, $p_2 \equiv 1$ mod 4)
and $(\frac{p_1}{p_2}) = +1$ satisfying the hypotheses in Theorem \ref{thm:def0rank4}.
\end{remark}

\begin{remark}
When the hypothesis $\text{Norm}_{k/\QQ} (\epsilon) = -1$ in Theorem \ref{thm:def0rank4} is not satisfied, 
the unit signature rank of $K$ is 3: the deficiency of the quadratic subfield $k$ is 1 by assumption, so the
deficiency of $K$ is at least 1 (hence the unit signature rank of $K$ is at most 3) 
and the first three rows in \eqref{eq:sigmatrix4} give a matrix of rank 3.
\end{remark}

We now consider fields of unit signature rank 3 in Proposition \ref{prop:equivalence}.

\begin{theorem} \label{thm:def0rank3}
Suppose $n > 1$ is an integer with $n \not\equiv 2$ mod 5 such that $n^2+1$ and 
$(n+1)^2 + 1$ are both squarefree.  Let $K = \QQ (\sqrt{n^2 + 1}, \sqrt{(n+1)^2 + 1})$.  Then each of the 
fundamental units $\epsilon_1$, $\epsilon_2$, and $\epsilon_3$ of the three quadratic 
subfields of $K$ has norm $-1$ and the unit signature rank of $K$ is 3: a set of fundamental units for
$K$ is given by $\{ \epsilon_1, \epsilon_2,  \epsilon_3 \}$.
\end{theorem}

\begin{proof}
For any positive integer $a$ such that $a^2 + 1$ is squarefree, the fundamental unit in the field
$\QQ(\sqrt{a^2 + 1})$ is $a + \sqrt{a^2 + 1}$, which has norm $-1$.  
We have $(n^2 + 1)[(n+1)^2 + 1] = N^2 + 1$ where $N = n(n+1) + 1$.  
Since $(2 n + 3)(n^2 + 1) - (2 n - 1) ((n+1)^2 + 1) = 5$, the greatest common divisor of
$n^2 + 1$ and $(n+1)^2 + 1$ divides 5, so equals 1 precisely when $n \not\equiv 2$ mod 5. 
Hence, if $n^2 + 1$ and $(n+1)^2 + 1$ are
squarefree and $n \not\equiv 2$ mod 5, then $N^2 + 1$ is also squarefree.  As a result, under
these hypotheses the three quadratic subfields of $K$ are $k_1 = \QQ(\sqrt{n^2 + 1})$,
$k_2 = \QQ(\sqrt{(n+1)^2 + 1})$ and $k_3 = \QQ(\sqrt{N^2 + 1})$, with 
fundamental units $\epsilon_1 = n + \sqrt{n^2 + 1}$, $\epsilon_2 = \sqrt{(n+1)^2 + 1}$, and 
$\epsilon_3 =  N + \sqrt{N^2 + 1}$, each with norm $-1$.    

By Propositions \ref{prop:unitstructure} and \ref{prop:equivalence} it remains to show 
$\eta = \epsilon_1 \epsilon_2 \epsilon_3$ is not a square in $K$.  As noted by Kubota
(see the discussion leading to equation \eqref{eq:conjugateeta}), if $\sqrt \eta \in K$,
then 
\begin{equation} \label{eq:norms}
\begin{aligned}
\nnnorm{K}{k_1}{\sqrt \eta} & = (-1)^{\nu_1} \epsilon_1 \\
\nnnorm{K}{k_2}{\sqrt \eta} & = (-1)^{\nu_2} \epsilon_2 \\
\nnnorm{K}{k_3}{\sqrt \eta} & = (-1)^{\nu_1 + \nu_2 + 1} \epsilon_3 \\
\end{aligned}
\end{equation}
for some integers $\nu_1, \nu_2 \in \{ 0, 1 \}$.

Writing $\eta = x + y \sqrt{n^2 + 1} + z \sqrt{(n+1)^2 + 1} + w  \sqrt{N^2 + 1} \in K$ with $N = n(n+1) + 1$ as above,
equation \eqref{eq:norms} gives the following six equations:
\begin{equation*} \label{eq:normsexpanded}
\begin{aligned}
(-1)^{\nu_1} n & = x^2 + y^2 - 2 z^2 - 2 w^2 - 2 z^2 n - 2 w^2 n + y^2 n^2 -  z^2 n^2 - 3 w^2 n^2 - 2 w^2 n^3 - w^2 n^4  \\
(-1)^{\nu_1} & = 2 x y - 4 z w - 4 z w n - 2 z w n^2 \\
(-1)^{\nu_2} (n+1) & =x^2 - y^2 + 2 z^2 - 2 w^2 + 2 z^2 n - 2 w^2 n - y^2 n^2 + z^2 n^2 - 3 w^2 n^2 - 2 w^2 n^3 - w^2 n^4 \\
(-1)^{\nu_2} & = 2 x z - 2 y w - 2 y w n^2 \\
%
(-1)^{\nu_1 + \nu_2 + 1} N & = x^2 - y^2 - 2 z^2 + 2 w^2 - 2 z^2 n + 2 w^2 n - y^2 n^2 - z^2 n^2 + 3 w^2 n^2 + 2 w^2 n^3 + w^2 n^4 \\
(-1)^{\nu_1 + \nu_2 + 1} & =  -2 y z + 2 x w . \\
\end{aligned}
\end{equation*}
Solving these equations yields the following values for $w$:
\begin{equation*} \label{eq:normsexpanded}
\begin{aligned}
\nu_1 = 0, \ \nu_2 = 0 : \quad & w = \pm n/(\sqrt{2}\sqrt{n^2 + 1}) \\
\nu_1 = 0, \ \nu_2 = 1 : \quad &  w = \pm 1/\sqrt{2} \\ 
\nu_1 = 1, \ \nu_2 = 0 : \quad &  w = N/ (\sqrt{2} \sqrt{N^2 + 1} ) \\
\nu_1 = 1, \ \nu_2 = 1 : \quad &  w = (n+1)/(\sqrt{2} \sqrt{(n+1)^2 + 1} ) \\
\end{aligned}
\end{equation*}
Because $n^2 + 1$, $(n+1)^2 + 1$ and $N^2 + 1$ are squarefree and greater than 2 (since $n >1$), it follows that
for each choice of $\nu_1, \nu_2 \in \{ 0, 1 \}$ there are no solutions \eqref{eq:norms} where $w$ is rational.
Since there are no solutions with rational $x,y,z,w$,
this proves $\sqrt{\epsilon_1  \epsilon_2  \epsilon_3}$ cannot be an element of $K$, completing the proof.
\end{proof}

\begin{remark}
When $n = 1$ the biquadratic field in Theorem \ref{thm:def0rank3} is $\QQ(\sqrt 2, \sqrt 5)$,
which has unit signature rank 4.  
\end{remark}

\begin{corollary}
There exist infinitely many real biquadratic fields $K$ whose quadratic subfields all have fundamental
units of norm $-1$ and where the unit signature rank of $K$ is 3.
\end{corollary}

\begin{proof}
The integers $n^2 + 1$ and $(n+1)^2 + 1$ are never divisible by 4 and if either is divisible by $p^2$ for
an odd prime $p$ then $p \equiv 1$ mod 4.  For each $p \equiv 1$ mod 4, $p > 5$, there are precisely two
residue classes $a$ mod $p^2$ for which $a^2 + 1 \equiv 0$ mod $p^2$ and two additional residue classes 
for which $(a + 1)^2 + 1 \equiv 0$ mod $p^2$.  When $p = 5$, there are precisely 7 residue classes $a$ mod 25
for which $a \equiv 2$ or $a^2 + 1 \equiv 0$ or $(a + 1)^2 + 1 \equiv 0$ mod 25.  It follows that the
number of $n \le x$ satisfying $n \not\equiv 2$ mod 5 and $n^2 + 1$ and $(n+1)^2 + 1$ both squarefree
is asymptotically equal to $C \sqrt x$ with $C = (1 - 7/25)\prod_{p \equiv 1 \mod 4, p > 5} ( 1 - 4/p^2 ) \sim 0.6810...$,
so there are infinitely many such $n$.  
\end{proof}

\section{Real biquadratic fields whose quadratic subfields all have deficiency 1}   \label{sec:deficiency1}

Suppose now that $K$ is a real biquadratic extension of $\QQ$ whose quadratic subfields $k_1$, $k_2$, and $k_3$ have corresponding
fundamental units $\epsilon_1$, $\epsilon_2$ and $\epsilon_3$ all with norm $+1$ (that is, each quadratic subfield has
deficiency 1).  

Since the deficiency in a finite extension of totally real fields never decreases, 
the deficiency of $K$ is at least 1 and no more than 3. The following Theorems \ref{thm:def1ranks2and3} and    
\ref{theorem:deficiency3} prove the existence of infinitely many examples of each of the three
possibilities for the unit signature rank.  The method of proof involves the explicit construction 
of the group of units of $K$ as in \cite{Kur}, using 
the `$m$-technology' of Proposition \ref{prop:m} as in \cite{Kub}, 
summarized in the following proposition.

\begin{proposition} \label{prop:Kubota}
Suppose $K$ is a real biquadratic field with quadratic subfields 
$k_1 = \QQ(\sqrt {d_1})$, $k_2 = \QQ(\sqrt {d_2})$ and $k_3 = \QQ(\sqrt {d_1 d_2})$ 
whose corresponding fundamental units are $\epsilon_1$, $\epsilon_2$ and $\epsilon_3$, each with
norm $+1$.  Let $E_K$ denote the group of units of $K$.  
If $m_i$ is the integer associated to $\epsilon_i$ in $k_i$ as in Proposition \ref{prop:m}, then
for integers $n_1,n_2,n_3$, the unit $ \epsilon_1^{n_1} \epsilon_2^{n_2} \epsilon_3^{n_3}$
is a square in $E_K$ if and only if the integer $m_1^{n_1} m_2^{n_2} m_3^{n_3}$ is one of
$ 1, d_1, d_2$, or $d_1 d_2 $ up to a rational square. 
\end{proposition}

\begin{proof}
Note that $m_i \cdot \epsilon_i$ is a square in $k_i$, hence is a square in $K$.  As a consequence, 
$ \epsilon_1^{n_1} \epsilon_2^{n_2} \epsilon_3^{n_3}$ is a square in $E_K$ if and only if $m_1^{n_1} m_2^{n_2} m_3^{n_3}$
is a square in $K$.  But an integer $m$ is a square in $K$ if and only if $K$ contains the field $\QQ(\sqrt{m})$, hence
if and only if $m$ differs from 1, $d_1$, $d_2$, or $d_1 d_2$ by a rational square, completing the proof.  
\end{proof}

By Proposition \ref{prop:unitstructure}, the unit group for $K$ is obtained from the group generated by 
$-1, \epsilon_1, \epsilon_2, \epsilon_3$ by extracting square roots of units 
$ \epsilon_1^{n_1} \epsilon_2^{n_2} \epsilon_3^{n_3}$ (namely, those that are squares in $K$).  For
this purpose it suffices to consider  $n_1,n_2,n_3 \in \{ 0,1 \}$, not all 0, and Proposition \ref{prop:Kubota}
provides a simple criterion to ascertain when the associated unit is a square, provided the integers
$m_1$, $m_2$, and $m_3$ are known.

\medskip

If $k = \QQ(\sqrt{q})$ for a prime $q \equiv 3$ mod 8, then the fundamental unit of $k$
necessarily has norm $+1$.  If $m$ is the corresponding positive integer given by
Proposition \ref{prop:m} then $m$ divides the discriminant, $4q$, of $k$, $m$ is squarefree,
and $m$ is neither equal to 1 nor to $q$.  Hence $m = 2$ or $m = 2 q$.
If $q \equiv 3$ mod 8, then 2 is not a square mod
$q$, hence 2 is not the norm of an integer from $k$, which implies $m = 2 q$.  If $q \equiv 7$ mod 8, then
$2 q$ is not the norm of an integer ($-q$ is a norm and $-2$ is not a square mod $q$, hence not a norm), 
so in this case $m = 2 $.

Suppose now that $K = \QQ( \sqrt{q_1}, \sqrt{q_2})$ with distinct primes $q_1$ and $q_2$ with
$q_1 \equiv q_2 \equiv 3 \mod 4$.
Let $k_1 = \QQ(\sqrt{q_1})$, $k_2 = \QQ(\sqrt{q_2})$, and $k_3 = \QQ(\sqrt{q_1 q_2})$ be the three
quadratic subfields of $K$, with corresponding fundamental units $\epsilon_1$, $\epsilon_2$ and 
$\epsilon_3$, respectively, each of which has norm $+1$ since every discriminant is divisible by a
prime $\equiv 3$ mod 4. Let $m_1$, $m_2$ and $m_3$ be the integers corresponding by Proposition \ref{prop:m} to 
$\epsilon_1$, $\epsilon_2$ and $\epsilon_3$, respectively.

The values of $m_1$ and $m_2$ are determined by the congruences of $q_1$ and $q_2$ mod 8, as above.

The possible values of $m_3$ are $q_1$ and $q_2$.  If $m_3 = q_1$, then $q_1$ is the norm of an integer from $\QQ(\sqrt{q_1 q_2})$
so $a^2 - q_1 q_2 b^2 = 4 q_1$ for some integers $a$ and $b$.  Reading this mod $q_2$ it follows that
$q_1$ is a quadratic residue mod $q_2$:
$( \frac{q_1}{q_2} ) = +1$.  If $m_3 = q_2$, then similarly $q_2$ is a quadratic residue mod $q_1$, 
that is, $( \frac{q_1}{q_2} ) = -1$
since $q_1$ and $q_2$ are $\equiv 3$ mod 4.  It follows that $m_3 = q_1$ if and only if $( \frac{q_1}{q_2} ) = +1$,
and $m_3 = q_2$ if and only if $( \frac{q_1}{q_2} ) = -1$.

\begin{theorem} \label{thm:def1ranks2and3} 
Suppose $K = \QQ( \sqrt{q_1}, \sqrt{q_2})$ for distinct primes $q_1$ and $q_2$ with $q_1 \equiv q_2 \equiv 3$ mod 4,
and with notation as above.  Then $\{ \sqrt{\epsilon_1 \epsilon_2 }, \sqrt{\epsilon_3},\epsilon_2 \}$ is a
set of fundamental units for $K$.  Furthermore, 
\begin{enumerate}

\item
if $q_1 \equiv q_2 \equiv 3$ mod 8, the unit signature rank of $K$ is 3, i.e., $K$ has deficiency 1,

\item
if $q_1 \equiv q_2 \equiv 7$ mod 8, the unit signature rank of $K$ is 2, i.e., $K$ has deficiency 2, and 

\item
if $q_1 \equiv 7$ and $ q_2 \equiv 3$ mod 8, the unit signature rank of $K$ is 3 and the deficiency is 1 
if $(q_1/q_2) = +1$ and both are equal to 2 if $(q_1/q_2) = -1$.

\end{enumerate}

\end{theorem}

\begin{proof}
By the remarks above, $m_1 = 2 q_1$ and $m_2 = 2 q_2$ in case (1),  $m_1 = m_2 = 2 $ in case (2), and
$m_1 = 2 $ and $m_2 = 2 q_2$ in case (3).  We have $m_3 = q_1$ if $( \frac{q_1}{q_2} ) = +1$, 
and $m_3 = q_2$ if $( \frac{q_1}{q_2} ) = -1$.

Suppose first that $q_1 \equiv q_2 \equiv 3$ mod 8.  Then $m_1 = 2 q_1$ and $m_2 = 2 q_2$.  
By Proposition \ref{prop:Kubota} we are interested in possible relations involving
$m_1^{n_1} m_2^{n_2} m_3^{n_3}$ mod squares, and since $q_1$ and $q_2$ differ by 
a square in $K$, for this question it is not important whether we take $m_3 = q_1$ or
$m_3 = q_2$.  For $n_1,n_2,n_3 \in \{0,1\}$, not all 0, the integer
$(2 q_1)^{n_1} (2 q_2)^{n_2} q_1^{n_3}$ differs by a square from one of $1$, $q_1$, $q_2$, $q_1 q_2$
precisely for $(n_1,n_2,n_3) = (1,1,0)$ or $(0,0,1)$ or $(1,1,1)$.  It follows that
a fundamental set of units for $K$ is given by 
$\{ \sqrt{\epsilon_1 \epsilon_2 }, \sqrt{\epsilon_3},\epsilon_2 \}$.

By Proposition \ref{prop:sqrtepsilon} 
\begin{equation*} \label{eq:sqrtepsilon1}
\sqrt{\epsilon_1} = \frac{1}{\sqrt{2 q_1}} \left( A_1  - B_1 \sqrt{q_1} \right)
\end{equation*}
with integers $A_1$, $B_1$ satisfying $A_1 > 0$, $B_1 < 0$ and $A_1^2 - q_1 B_1^2 = 2 q_1$.
Similarly,
\begin{equation*} \label{eq:sqrtepsilon2}
\sqrt{\epsilon_2} = \frac{1}{\sqrt{2 q_2}} \left( A_2  - B_2 \sqrt{q_2} \right)
\end{equation*}
with integers $A_2$, $B_2$ satisfying $A_2 > 0$, $B_2 < 0$ and $A_2^2 - q_2 B_2^2 = 2 q_2$.
Hence
\begin{equation} \label{eq:sqrtepsilon12}
\sqrt{\epsilon_1 \epsilon_2} =
\frac{1}{2 \sqrt{q_1 q_2} }\left( A_1  - B_1 \sqrt{q_1} \right)  \left( A_2  - B_2 \sqrt{q_2} \right).
\end{equation}
Since $A_1 > 0$ and $B_1 < 0$, we have $A_1 - B_1 \sqrt{q_1} > 0$.  Then $A_1^2 - q_1 B_1^2 = 2 q_1$ shows
$A_1 + B_1 \sqrt{q_1} > 0$ as well.  Similarly,
$A_2 - B_2 \sqrt{q_2}$ and $A_2 + B_2 \sqrt{q_2}$ are both positive.  

Let $\sigma \in \Gal(K/\QQ)$ be the element mapping $\sqrt{q_1}$ to $-\sqrt{q_1}$
and fixing $\sqrt{q_2}$ and let $\tau$ be the element mapping $\sqrt{q_2}$ to $-\sqrt{q_2}$
and fixing $\sqrt{q_1}$.

By \eqref{eq:sqrtepsilon12} it follows that

\begin{equation} \label{eq:signepsilon12}
\sigma ( \sqrt{\epsilon_1 \epsilon_2} ) < 0, \ \tau ( \sqrt{\epsilon_1 \epsilon_2} ) < 0, \text{ and } 
\sigma \tau ( \sqrt{\epsilon_1 \epsilon_2} ) > 0 .
\end{equation}

\smallskip

Since $m_3 = q_1$ if $( \frac{q_1}{q_2} ) = +1$ and $m_3 = q_2$ if $( \frac{q_1}{q_2} ) = -1$, then
as above, by Proposition \ref{prop:sqrtepsilon} we have

\begin{equation} \label{eq:sqrtepsilon3}
\sqrt{\epsilon_3} =
\begin{cases}
\dfrac{1}{\sqrt{q_1}} \left( A_3  - B_3 \sqrt{q_1 q_2} \right) , & \text{if } ( \dfrac{q_1}{q_2} ) = +1 , \\ 
\dfrac{1}{\sqrt{q_2}} \left( A_3  - B_3 \sqrt{q_1 q_2} \right) , & \text{if } ( \dfrac{q_1}{q_2} ) = -1 . \\
\end{cases}
\end{equation}

\smallskip
\noindent
with $A_3$, $B_3 \in (1/2) \ZZ$ such that 
$A_3 - B_3 \sqrt{q_1 q_2}$ and $A_3 + B_3 \sqrt{q_1 q_2}$ are both positive.  Hence

\begin{equation} \label{eq:signepsilon3}
\begin{cases}
\sigma ( \sqrt{\epsilon_3} ) < 0, \ \tau ( \sqrt{\epsilon_3} ) > 0, \ \sigma \tau ( \sqrt{\epsilon_3 } ) < 0 ,
  &  \text{if } ( \dfrac{q_1}{q_2} ) = +1 , \\[10 pt] 
\sigma ( \sqrt{\epsilon_3} ) > 0, \ \tau ( \sqrt{\epsilon_3} ) < 0, \ \sigma \tau ( \sqrt{\epsilon_3 } ) < 0 ,
  &  \text{if } ( \dfrac{q_1}{q_2} ) = -1 . \\ 
\end{cases}
\end{equation}

By \eqref{eq:signepsilon12} and \eqref{eq:signepsilon3}, the signature rank of  
$E_K = \gp{-1,\sqrt{\epsilon_1 \epsilon_2}, \sqrt{\epsilon_3}, \epsilon_2}$ is 3, which proves (1).

Suppose now that $q_1 \equiv q_2 \equiv 7$ mod 8.  Then $m_1 = m_2 = 2 $ and $m_3 = q_1$  or $q_2$.
As before, for $n_1,n_2,n_3 \in \{0,1\}$, not all 0, the integer
$2^{n_1} 2^{n_2} q_1^{n_3}$ differs by a square from one of $1$, $q_1$, $q_2$, $q_1 q_2$
precisely for $(n_1,n_2,n_3) = (1,1,0)$ or $(0,0,1)$ or $(1,1,1)$, hence again in this case
a set of fundamental units for $K$ is given by 
$\{ \sqrt{\epsilon_1 \epsilon_2 }, \sqrt{\epsilon_3},\epsilon_2 \}$.

In this case we have
\begin{equation*} \label{eq:sqrtepsilon12case2}
\sqrt{\epsilon_1 \epsilon_2} =
\frac{1}{2 }\left( A_1  - B_1 \sqrt{q_1} \right)  \left( A_2  - B_2 \sqrt{q_2} \right)
\end{equation*}
(where $A_1^2 - q_1 B_1^2 = A_2^2 - q_2 B_2^2 = 2$, $A_1,A_2 > 0$, $B_1,B_2 < 0$) and   
with $\sqrt{\epsilon_3}$ again given by \eqref{eq:sqrtepsilon3}.  In this case, however,
$\sqrt{\epsilon_1 \epsilon_2} $ is totally positive 
(since 
$A_1 - B_1 \sqrt{q_1}$, $A_1 + B_1 \sqrt{q_1}$, $A_2 - B_2 \sqrt{q_2}$ and $A_2 + B_2 \sqrt{q_2}$ are all
positive), so here
the signature rank of the units of $K$ is 2, which proves (2).

In case (3), $m_1 = 2$, $m_2 = 2 q_2$, with
$m_3 = q_1$ if $(q_1/q_2) = +1$
and
$m_3 = q_2$ if $(q_1/q_2) = -1$.
An analysis as before shows that $\{ \sqrt{\epsilon_1 \epsilon_2}, \sqrt{\epsilon_3}, \epsilon_2 \}$
again gives a set of fundamental units for $K$.  
Here $\sqrt{\epsilon_1 \epsilon_2}$ is $1/( \sqrt{2q_2} \sqrt{2}) = 1/(2\sqrt{q_2})$ times a totally positive element, 
and $\sqrt{\epsilon_3}$ is either 
$1/\sqrt{q_1}$ times a totally positive element if $(q_1/q_2) = +1$ or
$1/\sqrt{q_2}$ times a totally positive element if $(q_1/q_2) = -1$.

It follows that 
the unit signature rank of $K$ is 3 and the deficiency is 1 
if $(q_1/q_2) = +1$ and 
the unit signature rank and the deficiency 
of $K$ are both 2 if $(q_1/q_2) = -1$ (since $\sqrt{\epsilon_1 \epsilon_2}$ and $\sqrt{\epsilon_3}$ provide the
same signature in the latter case), which proves (3).
\end{proof}

\begin{remark} \label{rem:biquadprimesdeficiency}
Combined with results of the previous section, Theorem \ref{thm:def1ranks2and3} shows the unit signature
rank of a real biquadratic field $\QQ( \sqrt{l_1}, \sqrt{l_2} )$, where $l_1$ and $l_2$ are primes, is at least 2
(i.e., the deficiency is at most 2).
\end{remark}

\medskip

We now turn to the question of finding real biquadratic fields of deficiency 3.  Note that since the unit signature
rank never decreases in a totally real extension of a totally real field, each of the quadratic subfields of
a real biquadratic field of deficiency 3 must necessarily have deficiency 1.

\medskip

Suppose $q_1$, $q_2$, $q_3$ and $q_4$ are distinct primes $\equiv 3$ mod 4.  Then the fundamental unit
$\epsilon$ in $k = \QQ( \sqrt{ q_1 q_2 q_3 q_4} )$ has norm $+1$, and the associated $m = m_\epsilon$ is a
divisor of $q_1 q_2 q_3 q_4$, is neither equal to 1 nor to $q_1 q_2 q_3 q_4$, and is a norm
from $k$.   We consider the constraints on the values of the quadratic residue symbols
$(\frac{q_i}{q_j})$ imposed by different possibilities for $m$ in turn:

If $m = q_1$ then  $a^2 -  q_1 q_2 q_3 q_4 \cdot b^2 = 4 q_1$ implies $q_1$ divides $a$, so $a = q_1 a'$ and we have
$q_1 (a')^2 -  q_2 q_3 q_4 \cdot b^2 = 4 $.   Hence $q_1$ is a square mod $q_2$, $q_3$ and $q_4$: 
\begin{equation*}
\left( \frac{q_1}{q_2} \right) = \left( \frac{q_1}{q_3} \right) = \left( \frac{q_1}{q_4} \right) = +1 .
\end{equation*}
The corresponding constraints if $m = q_2$, after inverting one quadratic residue symbol, are
\begin{equation*}
\left( \frac{q_1}{q_2} \right) =  -1 , \quad \left( \frac{q_2}{q_3} \right) = \left( \frac{q_2}{q_4} \right) = +1 ,
\end{equation*}
with similar results if $m = q_3$ or $q_4$.  These are recorded in the first rows of Table \ref{table:quadresconstraints}.

If $m = q_2 q_3 q_4$, then $a^2 -  q_1 q_2 q_3 q_4 \cdot b^2 = 4 q_2 q_3 q_4$ implies $q_2 q_3 q_4$ divides $a$, so $a = q_2 q_3 q_4 a'$
and we have $q_2 q_3 q_4 (a')^2 - q_1 b^2 = 4$.  This is the same equation considered for $m = q_1$ except the signs are reversed,
so we obtain
\begin{equation*}
\left( \frac{q_1}{q_2} \right) = \left( \frac{q_1}{q_3} \right) = \left( \frac{q_1}{q_4} \right) = -1 ,
\end{equation*}
since the $q_i$ are $\equiv 3$ mod 4. The constraints for $m = q_1 q_2 q_3$, $q_1 q_2 q_4$ and $q_1 q_3 q_4$ are 
similarly the negatives of those considered
for $m$ a single prime and are recorded in the last rows of Table \ref{table:quadresconstraints}.

It remains to consider the cases when $m$ is the product of two primes.  If $m = q_1 q_2$, then as before we obtain the equation
$ q_1 q_2 (a')^2 - q_3 q_4 b^2 = 4$.  This yields
\begin{equation*}
\left( \frac{q_1 q_2}{q_3} \right) = \left( \frac{q_1 q_2}{q_4} \right) = +1  \text{ and }
\left( \frac{- q_3 q_4}{q_1} \right) = \left( \frac{- q_3 q_4}{q_2} \right) = +1 ,
\end{equation*}
so
\begin{equation*}
\left( \frac{q_1}{q_3} \right) = \left( \frac{q_2}{q_3} \right) , \quad
\left( \frac{q_1}{q_4} \right) = \left( \frac{q_2}{q_4} \right) ,  \text{ and }
\left( \frac{q_1}{q_3} \right) = -\left( \frac{q_1}{q_4} \right) , \quad
\left( \frac{q_2}{q_3} \right) = -\left( \frac{q_2}{q_4} \right)  .
\end{equation*}
Hence
\begin{equation*}
\left( \frac{q_1}{q_3} \right) = s, \quad \left( \frac{q_1}{q_4} \right) = -s, \quad
\left( \frac{q_2}{q_3} \right) = s, \quad \left( \frac{q_2}{q_4} \right) = -s
\end{equation*}
where $s = \pm 1$.  The constraints when $m$ is one of the other possible products of two primes are 
obtained similarly and all the results are recorded in the middle rows of Table \ref{table:quadresconstraints}.

Table \ref{table:quadresconstraints} summarizes the quadratic residue constraints imposed by the various possible 
values of $m = m_\epsilon$ by the condition that $m_\epsilon$ is a norm. 

\begin {table}[ht] 

\begin{center}
\begin{tabular}{l||c|c|c|c|c|c}
$m$ &
$\left( \frac{q_1}{q_2} \right)$  &  $\left( \frac{q_1}{q_3} \right)$  & $\left( \frac{q_1}{q_4} \right)$  
& $\left( \frac{q_2}{q_3} \right)$  & $\left( \frac{q_2}{q_4} \right)$  & $\left( \frac{q_3}{q_4} \right)$  \rule[-10pt]{0pt}{15pt}   \\
 \hline
 \hline

$q_1$ &  $+1$  & $+1$  &  $+1$  &   &   &    \\
$q_2$ &  $-1$  &   &   & $+1$  &  $+1$  &    \\
$q_3$ &    & $-1$  &   & $-1$  &    &   $+1$  \\
$q_4$ &    &   &  $-1$   &  &  $-1$   &   $-1$  \\
\hline
$q_1 q_2$ &    & $s$  & $-s$ & $s$  & $-s$   &    \\
$q_1 q_3$ & $s$   &   & $-s$ & $-s$  &    &  $-s$   \\
$q_1 q_4$ & $s$   & $-s$  &  &   & $-s$   &  $s$   \\
$q_2 q_3$ & $s$   & $s$  &  &   & $s$   &  $s$   \\
$q_2 q_4$ & $s$   &   & $s$ & $s$  &    &  $-s$   \\
$q_3 q_4$ &   & $s$   & $s$ & $-s$  & $-s$    &    \\
\hline
$q_1 q_2 q_3$ &   &     & $+1$  &    &  $+1$    &  $+1$  \\
$q_1 q_2 q_4$ &   &  $+1$   &    & $+1$   &     &  $-1$  \\
$q_1 q_3 q_4$ & $+1$  &     &    & $-1$   &  $-1$   &    \\
$q_2 q_3 q_4$ & $-1$  &  $-1$   &  $-1$   &    &    &    \\
\hline   
\end{tabular}
\end{center}
\caption {Summary of quadratic residue symbol constraints for primes $q_1$, $q_2$, $q_3$, $q_4$, each
$\equiv 3$ mod 4, if
the given value of $m$ is the norm of an integer from $\QQ( \sqrt{ q_1 q_2 q_3 q_4} )$.  Value of $s = \pm 1$ is 
fixed in any given row, but rows are independent}
\label{table:quadresconstraints} 

\end {table}

In the following theorem we use these constraints, together with the fact that there
exist infinitely many collections of primes satisfying the necessary quadratic residue 
relations by Dirichlet's theorem on primes in arithmetic progression, to prove there 
exist infinitely many real biquadratic fields of deficiency 3.

\begin{theorem} \label{theorem:deficiency3}
Suppose the primes $q_1, \dots, q_6$, each $\equiv 3$ mod 4, are chosen so that the following 
quadratic residue relations are satisfied:
\begin{equation*}
\left( \frac{q_1}{q_2} \right)   =  
\left( \frac{q_1}{q_3} \right)   = 
\left( \frac{q_1}{q_4} \right)   =  
\left( \frac{q_1}{q_5} \right)   = -1 , 
\left( \frac{q_1}{q_6} \right)   = 
\left( \frac{q_2}{q_3} \right)   = +1 , 
\left( \frac{q_2}{q_4} \right)   = -1 , 
\left( \frac{q_2}{q_5} \right)   = +1 , 
\end{equation*}
\begin{equation*}
\left( \frac{q_2}{q_6} \right)   = 
\left( \frac{q_3}{q_4} \right)   = +1 , 
\left( \frac{q_3}{q_5} \right)   = -1 , 
\left( \frac{q_3}{q_6} \right)   =  
\left( \frac{q_4}{q_5} \right)   = +1 , 
\left( \frac{q_4}{q_6} \right)   = 
\left( \frac{q_5}{q_6} \right)   = -1 . 
\end{equation*}

\smallskip
\noindent
Let $\epsilon_1$ denote the fundamental unit for $k_1 = \QQ( \sqrt{ q_1 q_2 q_3 q_4} )$,
$\epsilon_2$ the fundamental unit for $k_2 = \QQ( \sqrt{ q_1 q_2 q_5 q_6} )$, and
$\epsilon_3$ the fundamental unit for $k_3 = \QQ( \sqrt{ q_3 q_4 q_5 q_6} )$.  Then
$\{ \epsilon_1,\epsilon_2,\epsilon_3 \}$ is a set of fundamental units for the 
composite biquadratic field $K = \QQ( \sqrt{ q_1 q_2 q_3 q_4} , \sqrt{ q_1 q_2 q_5 q_6} )$.
In particular, these fundamental units are totally positive, so there exist infinitely 
many real biquadratic fields $K$ having deficiency 3, i.e.,
having unit signature rank 1.
\end{theorem}

\begin{proof}
Comparing the values of the various quadratic residue symbols for $q_1,q_2,q_3$ and $q_4$ in 
Table \ref{table:q1q2q3q4} to the results in Table \ref{table:quadresconstraints} 
shows that $q_2 q_3 q_4$ is the only possible value for
$m_1 = m_{\epsilon_1}$ for the field $\QQ( \sqrt{ q_1 q_2 q_3 q_4} )$.

\begin {table}[ht] 

\begin{center}
\begin{tabular}{c|c|c|c|c|c}
$\left( \frac{q_1}{q_2} \right)$  &  
$\left( \frac{q_1}{q_3} \right)$  & 
$\left( \frac{q_1}{q_4} \right)$  & 
$\left( \frac{q_2}{q_3} \right)$  & 
$\left( \frac{q_2}{q_4} \right)$  & 
$\left( \frac{q_3}{q_4} \right)$  \rule[-10pt]{0pt}{15pt}   \\
 \hline
 \hline
  $-1$ &  $-1$ &  $-1$ &  1 &  $-1$ &  1 \\
\hline   
\end{tabular}
\end{center}
\caption {Quadratic residue symbol values for $q_1$, $q_2$, $q_3$, $q_4$ in Theorem \ref{theorem:deficiency3}.}
\label{table:q1q2q3q4} 

\end {table}

Comparing the values of the various quadratic residue symbols for $q_1,q_2,q_5$ and $q_6$ in 
Table \ref{table:q1q2q5q6} 
to the results in Table \ref{table:quadresconstraints} (with
$q_3$ replaced by $q_5$ and $q_4$ replaced by $q_6$) shows that $q_2$ is the only possible value for
$m_2 = m_{\epsilon_2}$  for the field $\QQ( \sqrt{ q_1 q_2 q_5 q_6} )$.

\begin{table}[ht] 

\begin{center}
\begin{tabular}{c|c|c|c|c|c}
$\left( \frac{q_1}{q_2} \right)$  &  
$\left( \frac{q_1}{q_5} \right)$  & 
$\left( \frac{q_1}{q_6} \right)$  & 
$\left( \frac{q_2}{q_5} \right)$  & 
$\left( \frac{q_2}{q_6} \right)$  & 
$\left( \frac{q_5}{q_6} \right)$  \rule[-10pt]{0pt}{15pt}   \\
 \hline
 \hline
$-1$ &  $-1$ &  1 &  1 &  1 &  $-1$  \\
\hline   
\end{tabular}
\end{center}
\caption {Quadratic residue symbol values for $q_1$, $q_2$, $q_5$, $q_6$ in Theorem \ref{theorem:deficiency3}.}
\label{table:q1q2q5q6} 

\end {table}
\noindent

Finally, comparing the values of the various quadratic residue symbols for  $q_3,q_4,q_5$ and $q_6$ in 
Table \ref{table:q3q4q5q6} 
to the results in Table \ref{table:quadresconstraints} (with $q_1$, $q_2$, $q_3$ and $q_4$ replaced by 
$q_3$, $q_4$, $q_5$ and $q_6$, respectively)
shows that $q_4 q_6$ is the only possible value for
$m_3 = m_{\epsilon_3}$  for the field $\QQ( \sqrt{ q_3 q_4 q_5 q_6} )$.

\begin {table}[ht] 

\begin{center}
\begin{tabular}{c|c|c|c|c|c}
$\left( \frac{q_3}{q_4} \right)$  &  
$\left( \frac{q_3}{q_5} \right)$  & 
$\left( \frac{q_3}{q_6} \right)$  & 
$\left( \frac{q_4}{q_5} \right)$  & 
$\left( \frac{q_4}{q_6} \right)$  & 
$\left( \frac{q_5}{q_6} \right)$  \rule[-10pt]{0pt}{15pt}   \\
 \hline
 \hline
1 &  $-1$ &  1 &  1 &  $-1$ &  $-1$  \\
\hline   
\end{tabular}
\end{center}
\caption {Quadratic residue symbol values for $q_3$, $q_4$, $q_5$, $q_6$ in Theorem \ref{theorem:deficiency3}.}
\label{table:q3q4q5q6} 

\end {table}
Since $m_1 = q_2 q_3 q_4$,
$m_2 = q_2$, and 
$m_3 =  q_4 q_6$, if not all of $n_1$, $n_2$ and $n_3$ are even, then
$m_1^{n_1}  m_2^{n_2}  m_3^{n_3}$ is (up to the square of a rational integer) one of 
$q_2$, $q_3 q_4$, $q_2 q_3 q_4$, $q_3 q_6$, $q_2 q_3 q_6$, $q_4 q_6$, or $q_2 q_4 q_6$.
Because none of these is $1,  q_1 q_2 q_3 q_4, q_1 q_2 q_5 q_6$ or $q_3 q_4 q_5 q_6$, 
it follows by Proposition \ref{prop:Kubota} that 
$\epsilon_1^{n_1}  \epsilon_2^{n_2}  \epsilon_3^{n_3}$ is a square in $K$
only when $n_1$, $n_2$ and $n_3$ are all even.  This proves 
$\{ \epsilon_1,\epsilon_2,\epsilon_3 \}$ is a set of fundamental units for $K$.  

Using Dirichlet's theorem on primes in arithmetic progression, choosing any prime 
$q_1 \equiv 3$ mod 4 and then the primes $q_i$, $i = 2, \dots , 6$ inductively to satisfy
the congruences necessary for the required quadratic residue relations with respect to the
primes $q_j$, $j < i$, constructs infinitely many such real biquadratic fields, completing the
proof. 
\end{proof}

\begin{example}
For an explicit example: take $q_1 = 31$, $q_2 = 47$, $q_3 = 67$, $q_4 = 7$, $q_5 = 19$, $q_6 = 11$, with
$q_1 q_2 q_3 q_4 =  683333$, $q_1 q_2 q_5 q_6 =  304513$, 
$q_3 q_4 q_5 q_6 =   98021$ and asociated 
quadratic residue data $( -1,-1,-1,-1,1,1,-1,1,1,1,-1,1,1,-1,-1 )$ as in  
Theorem \ref{theorem:deficiency3}.  The corresponding biquadratic field $K = \QQ(\sqrt{683333} , \sqrt{304513})$ has
class group  isomorphic to $ (\ZZ/ 2 \ZZ)^2 \times (\ZZ/4\ZZ)$ 
and strict class group  isomorphic to $ (\ZZ/ 2 \ZZ)^3 \times (\ZZ/4\ZZ)^2$ 
\end{example}

\begin{remark}
There are many quadratic residue configurations $( \dots, (\frac{q_i}{q_j}) , \dots )$, $ 1 \le i < j \le 6$ 
for which $\{ \epsilon_1,\epsilon_2,\epsilon_3 \}$ is a set of fundamental units for the 
composite biquadratic field 
$K$
as in Theorem \ref{theorem:deficiency3}.
For example, $(-1, -1, -1, -1, 1, -1, -1, -1, -1, -1, -1, -1, -1, 1, -1)$
is another instance---an explicit example of which is given by $q_1 =  43$, $q_2 = 11$,  $q_3 = 23$,  $q_4 = 31$,  $q_5 = 47$,  $q_6 = 3$, so that
$K = \QQ( \sqrt{337249}, \sqrt{66693})$, whose class group is
isomorphic to $ (\ZZ/ 2 \ZZ)^2 \times (\ZZ/24\ZZ)$ and whose strict class group is
isomorphic to $ (\ZZ/ 2 \ZZ)^3 \times (\ZZ/4\ZZ) \times (\ZZ/24\ZZ)$.
In fact, the method of proof of Theorem \ref{theorem:deficiency3} applies to 14080 of the $2^{15}$ total
possible quadratic residue configurations.
\end{remark}

\begin{remark}
By \cite{D}, the strict class group of the biquadratic field $K$ in Theorem \ref{theorem:deficiency3} always contains
at least one element of order 4.  
\end{remark}

\section{Real multiquadratic fields}  

As previously noted, the unit signature rank of a real biquadratic field $K$
is influenced both by the signatures of the fundamental units of its three quadratic 
subfields and by how those fundamental units are situated in the units of $K$.  
As we now show, the signatures provided by the units from the 
quadratic subfields provide relatively few 
signatures in higher rank multiquadratic fields.

We first observe that units from real quadratic subfields define characters of Galois extensions of $\QQ$,
a result that may be of independent interest.

\begin{lemma} \label{lem:signcharacters}
Suppose that $L$ is any Galois extension of $\QQ$ contained in $\CC$ 
and that $\epsilon > 0$ is a unit in a real quadratic
subfield of $L$.
Then the map $\chi_\epsilon$ defined by $\chi_\epsilon (\sigma) = \text{sign}(\sigma(\epsilon))$
defines a character of $\Gal(L/\QQ)$.
\end{lemma}

\begin{proof}
Since $\sigma (\epsilon) = \text{sign}(\sigma(\epsilon)) / \epsilon$ if $\sigma$
acts nontrivially on $\epsilon$, we have 
$\sigma (\epsilon) = \text{sign}(\sigma(\epsilon)) \, \epsilon^{\pm 1}$ for all $\sigma$.  Then
for any $\sigma, \tau \in \Gal(L/\QQ)$,
\begin{equation*}
\sigma \tau (\epsilon) = \sigma \left (  \text{sign}(\tau(\epsilon)) \ \epsilon^{\pm 1}  \right ) 
=
\text{sign}(\tau(\epsilon)) \text{ sign}(\sigma(\epsilon)) \ \epsilon^{\pm 1}
\end{equation*}
from which it follows that
\begin{equation*}
\chi_\epsilon (\sigma \tau) =  \text{ sign}(\sigma \tau (\epsilon)) =
\text{ sign}(\sigma(\epsilon))  \text{ sign}(\tau(\epsilon))= \chi_\epsilon (\sigma) \chi_\epsilon (\tau) ,
\end{equation*}
so that $\chi_\epsilon$ is a character of $\Gal(L/\QQ)$.
\end{proof}

Suppose in particular that $L$ is a real multiquadratic extension of rank $t$, that is, $L$ is a Galois
extension of $\QQ$ contained in the reals with $\Gal(L/\QQ) \iso (\ZZ / 2 \ZZ)^t$.  The field $L$
contains precisely $2^t - 1$ real quadratic subfields.  For
$1 \le i \le 2^t - 1$ let $\epsilon_i$ denote the fundamental unit from the quadratic subfield $k_i$
of $L$ and assume $\nnorm{k_i}{\epsilon} = -1$ for every $i$.

The signature of $\epsilon_i$, considered as an element of $L$,
consists of $2^{t-1}$ values of 1 (or 0 when viewed additively) for the automorphisms $\sigma \in \Gal(L/k_i) \le \Gal(L/\QQ)$
that fix $k_i$ and $2^{t-1}$ values of $-1$ (or 1 when viewed additively) for the nontrivial coset of $\Gal(L/k_i)$ in $\Gal(L/\QQ)$. 
If $\chi_i = \chi_{\epsilon_i}$ denotes the character given by the signatures of $\epsilon_i$ as in 
Lemma \ref{lem:signcharacters}, then $\chi_i$ is a character of order 2 whose kernel is $\Gal(L/k_i)$.
Since the subgroups $\Gal(L/k_i)$ for $1 \le i \le 2^t - 1$ are all the distinct subgroups
of $G$ of index 2, it follows that the $ \chi_i $ for $1 \le i \le 2^t - 1$ give 
all the nontrivial characters of $G$.  

Viewing the characters $\chi_i$ as having additive values in $\FF_2$, 
it follows that the $2^t \times 2^t$ matrix over $\FF_2$ of signatures of $\{ -1, \epsilon_1, \dots, \epsilon_{2^t - 1} \}$ 
is the character table for $G$ except that 
the row for the trivial character of $G$ is replaced by a row consisting of $1$'s (the signature of
$-1$ viewed additively).  

As a result, the rank of this matrix is $t + 1$:  the character group is isomorphic to $(\ZZ / 2\ZZ)^t$, so is generated
by $t$ elements, and the row of $2^t$ 1's is independent of the remaining rows (all of which have a
0 in the column corresponding to the identity element in $G$).

If any of the units $\epsilon_i$ from the quadratic subfields of $L$ has norm $+1$, then the rank of the group
of signatures can only decrease.  We summarize this in the following proposition.

\begin{proposition} \label{prop:multiquadratic}
The maximum possible signature rank coming from the units 
in the $2^t - 1$ real quadratic subfields of a real multiquadratic field $L$ of rank $t$ is $t+1$.
\end{proposition}

\begin{remark}
The field $L = \QQ( \sqrt{p_1}, \dots , \sqrt{p_t})$ for primes $p_1, \dots , p_t$ all $\equiv 1$ mod 4
shows the maximum value of $t+1$ can be achieved.  For a 
biquadratic field, this maximum is 3, as noted in equation \eqref{eq:sigmatrix4}.
\end{remark}

\medskip

The following result shows that for $(\ZZ / 2 \ZZ)^3$ extensions generated by the square roots of primes
it is not possible to have a unit signature rank less than 3 (compare to Remark \ref{rem:biquadprimesdeficiency}
that the analogous $(\ZZ/2\ZZ)^2$ extensions have signature rank at least 2).

\begin{proposition}
Suppose $K = \QQ(\sqrt{l_1}, \sqrt{l_2}, \sqrt{l_3})$ with $l_1, l_2, l_3$ distinct primes.  Then the 
unit signature rank of $K$ is at least 3.  This minimum is best possible: the field 
$\QQ(\sqrt{7}, \sqrt{23}, \sqrt{127})$ has deficiency 5.  All other possible deficiencies between 0 and 5 
can occur: $\QQ(\sqrt{5}, \sqrt{13}, \sqrt{37} )$ has deficiency 0, 
$\QQ(\sqrt{5}, \sqrt{13}, \sqrt{17} )$ has deficiency 1,  
$\QQ(\sqrt{5}, \sqrt{13}, \sqrt{29} )$ has deficiency 2,
$\QQ(\sqrt{3}, \sqrt{7}, \sqrt{11} )$ has deficiency 3, and  $\QQ(\sqrt{7},  \sqrt{23}, \sqrt{71})$ 
has deficiency 4.
\end{proposition}

\begin{proof}
Suppose that either $l_1 = 2$ or $l_1 \equiv 1$ mod 4 and that $l_2 \equiv 1$ mod 4.  
Then the unit signature rank of the biquadratic field $\QQ(\sqrt{l_1}, \sqrt{l_2})$ is at least 3 from the first three rows
of \eqref{eq:sigmatrix4}, so $K$ also has at least 3 independent unit signatures.
Suppose next that either $l_1 = 2$ or $l_1 \equiv 1$ mod 4 and that $l_2 \equiv l_3 \equiv 3$ mod 4.
Then the biquadratic field $\QQ(\sqrt{l_2}, \sqrt{l_3})$ has unit signature rank at least 2 by 
Remark \ref{rem:biquadprimesdeficiency}, and the fundamental unit from $\QQ(\sqrt{l_1})$ provides an independent signature,
so again the unit signature rank of $K$ is at least 3.

It remains to consider the case when $l_1 \equiv l_2 \equiv l_3 \equiv 3$ mod 4.
If at least two of $l_1, l_2, l_3$ are congruent to 3 mod 8, then (1) of Theorem \ref{thm:def1ranks2and3} 
implies the unit signature rank of $K$ is at least 3. 
Suppose finally that $l_1 \equiv l_2 \equiv 7$ mod 8 and order $l_1$, $l_2$ so that $(l_1/l_2) = +1$.
Then as in the proof of (2) of Theorem \ref{thm:def1ranks2and3}, the square root of the
fundamental unit of $\QQ(\sqrt{l_1 l_2})$ lies in $K$ and differs from $1/\sqrt{l_1}$ by a totally positive element.
Similarly (from the proof of (2) of Theorem \ref{thm:def1ranks2and3} if $l_3 \equiv 7$ mod 8, and by 
the proof of (3) of Theorem \ref{thm:def1ranks2and3}
if $l_3 \equiv 3$ mod 8), the square root of the fundamental unit of $\QQ(\sqrt{l_2 l_3})$ lies in $K$ and differs from either
$1/\sqrt{l_2}$ or $1/\sqrt{l_3}$ (according as $(l_2/l_3) = +1$ or $(l_2/l_3) = -1$, respectively)
by a totally positive element. These units together with $-1$ show that $K$ contains units with at least 
3 independent signatures.  Finally, a computation shows $\QQ(\sqrt{7}, \sqrt{23}, \sqrt{127})$ has
unit signature rank 3, with similar computations for the remaining fields, completing the proof.  
\end{proof}

Although by the previous proposition it is not possible to find a real $(\ZZ/2\ZZ)^3$-extension with unit signature rank less than 3 
that is generated by
square roots of {\it primes}, it is possible to prove the existence of infinitely many real $(\ZZ/2\ZZ)^3$-extensions having
unit signature rank 1, using the same sort of techniques as in the proof of Theorem \ref{theorem:deficiency3}.
We first make a general remark about certain norms from quadratic extensions.  

Suppose $m = \prod_{i \in S_1} q_i$ is the norm of an integer from $\QQ(\sqrt{q_1 \cdots q_n})$ with $n$ even, with $q_1, \dots , q_n$ 
distinct primes each congruent to 3 mod 4, and with 
$S_1$ a nonempty proper subset of $\{1,2,\dots, n\}$.  Let
$S_2$ be the complement of $S_1$, so $S_1 \sqcup S_2 = \{1,2,\dots, n\}$ (disjoint union).  Then
since $m$ is the norm of an integer, 
\begin{equation*}
a^2 - (q_1 \cdots q_n) b^2 = \prod_{i \in S_1} q_i ,
\end{equation*}
with integers $a$ and $b$, which implies
\begin{equation*}
\left( \prod_{i \in S_1} q_i \right ) a'^2 - \left( \prod_{i \in S_2} q_i \right ) b^2 = 1,
\end{equation*}
with integers $a'$ and $b$.

This gives the following two collections of quadratic residue constraints:

\begin{equation} \label{eq:firstconstraint}
\left( 
\dfrac 
{ \prod_{i \in S_1} q_i }
{ q_j }
\right ) = 1, 
\quad \text{i.e., } \quad
\prod_{i \in S_1} \left(  \dfrac{q_i}{q_j}  \right ) = 1
\qquad \text{for every } j \in S_2
\end{equation}
and
\begin{equation} \label{eq:secondconstraint}
\left( 
\dfrac 
{ \prod_{j \in S_2} q_j }
{ q_i }
\right ) = -1, 
\quad \text{i.e., } \quad
\prod_{j \in S_2} \left(  \dfrac{q_j}{q_i}  \right ) = -1
\qquad \text{for every } i \in S_1 .
\end{equation}

\begin{theorem} \label{thm:rank3totpos}
Let $K = \QQ( \sqrt{q_1 q_2 q_3 q_4}, \sqrt{q_1 q_2 q_5 q_6}, \sqrt{q_1 q_2 q_7 q_8} )$ where $q_1, \dots , q_8$ are
distinct primes, each congruent to 3 mod 4, and chosen so that the following quadratic residue relations are satisifed:
\begin{equation*}
\left( \frac{q_1}{q_2} \right)   =  
\left( \frac{q_1}{q_3} \right)   = 
\left( \frac{q_1}{q_4} \right)   =  
\left( \frac{q_1}{q_5} \right)   = -1 , 
\left( \frac{q_1}{q_6} \right)   = +1 ,
\left( \frac{q_1}{q_7} \right)   = -1 ,
\left( \frac{q_1}{q_8} \right)   = 
\left( \frac{q_2}{q_3} \right)   = +1 , 
\end{equation*}
\begin{equation*}
\left( \frac{q_2}{q_4} \right)   = -1 , 
\left( \frac{q_2}{q_5} \right)   =  
\left( \frac{q_2}{q_6} \right)   = +1 ,
\left( \frac{q_2}{q_7} \right)   = 
\left( \frac{q_2}{q_8} \right)   = -1 ,
\left( \frac{q_3}{q_4} \right)   = +1 , 
\left( \frac{q_3}{q_5} \right)   = -1 , 
\end{equation*}
\begin{equation*}
\left( \frac{q_3}{q_6} \right)   = +1 ,
\left( \frac{q_3}{q_7} \right)   = 
\left( \frac{q_3}{q_8} \right)   = -1 ,
\left( \frac{q_4}{q_5} \right)   = +1 , 
\left( \frac{q_4}{q_6} \right)   = -1 ,
\left( \frac{q_4}{q_7} \right)   = 
\left( \frac{q_4}{q_8} \right)   = +1 ,
\end{equation*}
\begin{equation*}
\left( \frac{q_5}{q_6} \right)   = -1 ,
\left( \frac{q_5}{q_7} \right)   = +1 ,
\left( \frac{q_5}{q_8} \right)   = -1 ,
\left( \frac{q_6}{q_7} \right)   = 
\left( \frac{q_6}{q_8} \right)   = 
\left( \frac{q_7}{q_8} \right)   = +1  .
\end{equation*}
Then $\{ \epsilon_1, \dots , \epsilon_7 \}$ is a set of totally positive fundamental units for $K$, where the
$\epsilon_i$, $1 \le i \le 7$, are the fundamental units for the seven quadratic subfields of $K$.  In particular, the
unit signature rank of $K$ is 1, and there are infinitely many such fields $K$. 

\end{theorem}

\begin{proof}
Using Table \ref{table:quadresconstraints}, we see that the values of $m_i = m_{\epsilon_i}$ for six of the quadratic 
subfields $k_i$ of $K$ are as in Table \ref{table:rank3mvalues}

\begin {table}[ht] 

\begin{center}
\begin{tabular}{c|l}
$k_i$ & \quad $m_i = m_{\epsilon_i}$ \\
 \hline
$\QQ ( \sqrt{q_1 q_2 q_3 q_4} )$   \rule[0pt]{0pt}{12pt}  &  \quad $m_1 = q_2 q_3 q_4$   \\[0.1cm]
$\QQ ( \sqrt{q_1 q_2 q_5 q_6} )$   \rule[0pt]{0pt}{12pt}  &  \quad $m_2 = q_2$   \\[0.1cm]
$\QQ ( \sqrt{q_3 q_4 q_5 q_6} )$   \rule[0pt]{0pt}{12pt}  &  \quad $m_3 = q_4 q_6$   \\[0.1cm]
$\QQ ( \sqrt{q_1 q_2 q_7 q_8} )$   \rule[0pt]{0pt}{12pt}  &  \quad $m_4 = q_7$   \\[0.1cm]
$\QQ ( \sqrt{q_3 q_4 q_7 q_8} )$   \rule[0pt]{0pt}{12pt}  &  \quad $m_5 = q_7 q_8$   \\[0.1cm]
$\QQ ( \sqrt{q_5 q_6 q_7 q_8} )$   \rule[0pt]{0pt}{12pt}  &  \quad $m_6 = q_6$   
\end{tabular}
\end{center}
\caption { }
\label{table:rank3mvalues} 

\end {table}

Using the values of the quadratic residue symbols $(q_i/q_j)$, $1 \le i < j \le 8$, together with equations
\eqref{eq:firstconstraint} and  \eqref{eq:secondconstraint} shows that the only possibility for
$m_7$, the value of $m_{\epsilon}$ for the remaining quadratic subfield
$\QQ ( \sqrt{q_1 q_2 q_3 q_4 q_5 q_6 q_7 q_8 } )$, is $q_1 q_6 q_7$, $q_3 q_5 q_7$ or $q_1 q_3 q_5 q_6$.
An examination of the values 
\begin{equation*}
(q_2 q_3 q_4)^{n_1}
q_2^{n_2}
(q_4 q_6)^{n_3}
q_7^{n_4}
(q_7 q_8)^{n_5}
q_6^{n_6}
\left \{ 
\begin{aligned}
& (q_1 q_6 q_7)^{n_7} \\
& (q_3 q_5 q_7)^{n_7} \\
& (q_1 q_3 q_5 q_6)^{n_7}
\end{aligned}
\right \}
\end{equation*}
for $n_i$ equal to 0 or 1, $1 \le i \le 7$, shows that only $n_1 = n_2 = \cdots = n_7 = 0$ gives one of the
values  1, $q_1 q_2 q_3 q_4$, $q_1 q_2 q_5 q_6$, $q_3 q_4 q_5 q_6$, $q_1 q_2 q_7 q_8$, 
$q_3 q_4 q_7 q_8$, $q_5 q_6 q_7 q_8$, or $q_1 q_2 q_3 q_4 q_5 q_6 q_7 q_8$ up to a square.  It follows that
there are no quadratic relations among the fundamental units $\epsilon_i$, $ 1 \le i \le 7$.  
Since the quotient of $E_K$ by the group $\gp {-1, \epsilon_1, \dots , \epsilon_7 }$ is a finite abelian 2-group, the
fact that there are no elements of order 2 in this quotient shows that $E_K = \gp {-1, \epsilon_1, \dots , \epsilon_7 }$.
As before, choosing primes $q_1, \dots , q_8$ inductively using Dirichlet's theorem on primes in arithmetic
progressions shows there are infinitely many such triquadratic fields $K$.
\end{proof}

\begin{example}
An explicit example of a field $K$ as in Theorem \ref{thm:rank3totpos} is given by taking
$q_1 = 11 $, $q_2 = 67$, $q_3 = 991$,
$q_4 = 47$, $q_5 = 31$, $q_6 = 7$,
$q_7 = 199$, $q_8 = 19$. In this case, the value of $m_{\epsilon}$ for the quadratic subfield
$\QQ ( \sqrt{q_1 q_2 q_3 q_4 q_5 q_6 q_7 q_8 } )$, is $q_3 q_5 q_7= 6113479 $ (here the fundamental unit $\epsilon$ has size
roughly $1.76 (10^{8153})$).
\end{example}

We next consider the question of finding multiquadratic extensions having maximal rather than minimal unit signature rank.  
If any of the quadratic subfields of a multiquadratic field has a fundamental unit whose norm is $+1$ then the deficiency
of the multiquadratic field must be at least 1.  As a result, a precursor to finding a multiquadratic field of 
rank $t$ with the maximum possible unit signature rank of $2^t$ would be to find a multiquadratic field all of whose quadratic subfields
have fundamental units of norm $-1$. The following result provides infinitely many examples of such
fields when $t = 3$.

\begin{proposition} \label{prop:rnk3unitsminus1}
Suppose that $p_1, p_2,p_3$ are distinct primes with
$p_2 \equiv p_3 \equiv 1$ mod 4 and with $p_1 = 2$ or $p_1 \equiv 1$ mod 4,
and suppose further that at
least two of $\left( \frac{p_1}{p_2} \right)$, $\left( \frac{p_1}{p_3} \right)$, and
$\left( \frac{p_2}{p_3} \right)$ are equal to $-1$.  
Then the fundamental unit in $\QQ(\sqrt{p_1 p_2 p_3} )$ has norm $-1$.
In particular, if $(p_1/p_2) = (p_1/p_3) = (p_2/p_3) = -1$ then the fundamental unit in every quadratic
subfield of $L = \QQ( \sqrt{p_1}, \sqrt{p_2}, \sqrt{p_3})$ has norm $-1$.  
\end{proposition}

\begin{proof}
Suppose the fundamental unit $\epsilon$ in $\QQ(\sqrt{p_1 p_2 p_3})$ has norm $+1$.  Then the possible values for
$m_\epsilon$ from Proposition \ref{prop:m} are $p_1,p_2,p_3$ or $p_1 p_2$, $p_1 p_3$, $p_2 p_3$.  
If $m = p_1$, then there are integers $a,b$ with $a^2 -  p_1 p_2 p_3 \cdot b^2 = 4 p_1$
(or $a^2 - 2 p_2 p_3 b^2 = 2$ if $p_1 = 2$), so integers $a',b$ with $p_1 \cdot (a')^2 -  p_2 p_3 \cdot b^2 = 4 $ 
(or $2 (a')^2 - p_2 p_3 b^2 = 1$ if $p_1 = 2$), hence
\begin{equation}
\left( \frac{p_1}{p_2} \right) = \left( \frac{p_1}{p_3} \right) =  +1 ,
\end{equation}
with similar statements for $m = p_2$ and $m = p_3$. 
If $m = p_1 p_2$, then there are integers with $a^2 -  p_1 p_2 p_3 \cdot b^2 = 4 p_1 p_2$
(or $a^2 -  2 p_2 p_3 \cdot b^2 = 2 p_2$ if $p_1 = 2$),
so integers with $p_1 p_2\cdot (a')^2 -  p_3 \cdot b^2 = 4$ ($2 p_2\cdot (a')^2 -  p_3 \cdot b^2 = 1$ if $p_1 = 2$), which
give the same conditions as $m = p_3$.  Again there are similar statements for $m = p_1 p_3$ and $m = p_2 p_3$.  
This information is summarized in Table \ref{table:quadresconstraints3primes}.

\begin {table}[ht] 

\begin{center}
\begin{tabular}{l||c|c|c}
$m$ &
$\left( \frac{p_1}{p_2} \right)$  &  $\left( \frac{p_1}{p_3} \right)$  & $\left( \frac{p_2}{p_3} \right)$   \rule[-10pt]{0pt}{15pt}   \\
 \hline
 \hline
$p_1$ &  $+1$  & $+1$  &    \\
$p_2$ &  $+1$ &       &   $+1$     \\
$p_3$ &    &  $+1$  &  $+1$ \\
\hline
$p_1 p_2$ &    & $+1$  &  $+1$ \\
$p_1 p_3$ &  $+1$ &       &   $+1$     \\
$p_2 p_3 $ &  $+1$  & $+1$  &    \\
\hline   
\end{tabular}
\end{center}
\caption {
Summary of quadratic residue symbol constraints for distinct primes $p_1$, $p_2$, $p_3$, 
with $p_2 \equiv p_3 \equiv 1$ mod 4 and with $p_1 = 2$ or $p_1 \equiv 1$ mod 4,
if the given value of $m$ is the norm of an integer from $\QQ( \sqrt{ p_1 p_2 p_3 } )$.
}
\label{table:quadresconstraints3primes} 

\end {table}

As a consequence, for any $p_1,p_2,p_3$ with at least two of $\left( \frac{p_1}{p_2} \right)$, $\left( \frac{p_1}{p_3} \right)$, and
$\left( \frac{p_2}{p_3} \right)$ equal to $-1$, then  
all of these conditions fail, so the fundamental unit in $\QQ(\sqrt{p_1 p_2 p_3})$ has norm $-1$, which proves the
first statement in the proposition.  Combining this with Propositions \ref{prop:p} and \ref{prop:p1p2} proves the
remaining statement.
\end{proof}

The result in Proposition \ref{prop:rnk3unitsminus1} follows from a result of R\'edei (see \cite[Proposition 4.1]{St}
and Stevenhagen's comments regarding it).  The method used for the simple 
proof above provides an equally elementary proof of the following result (that
also follows from R\'edei's theorem) generalizing the second statement in Proposition \ref{prop:rnk3unitsminus1}.

\begin{proposition} \label{prop:todd}
Suppose that $p_1, \dots, p_t$ are distinct primes with $p_2 \equiv p_3 \equiv \dots \equiv p_t \equiv 1$ mod 4 and
with $p_1 = 2$ or $p_1 \equiv 1$ mod 4.  If $t$ is odd, and $(p_i/p_j) = -1$ for every $i < j$ (i.e.,
$(p_i/p_j) = -1$ for all $i$ and $j$ with the convention that $(p/2) = (2/p)$ if $p_1 = 2$), then 
the fundamental unit in $\QQ(\sqrt{p_1 p_2 \dots p_t} )$ has norm $-1$.
\end{proposition}

\begin{proof}
Suppose by way of contradiction that the norm of the fundamental unit is $+1$ and let $S$ be the nonempty proper
subset of $\{1,2,\dots, t\}$ with $m= \prod_{i \in S} p_i$ for the
associated integer from Proposition \ref{prop:m}.  

Since $m$ is the norm of an integer from $\QQ(\sqrt{p_1 \cdots p_t})$, then just as in the derivation of
equations \eqref{eq:firstconstraint} and \eqref{eq:secondconstraint} it follows that
\begin{equation} \label{eq:p1mod4constraint1}
\prod_{i \in S} \left(  \dfrac{p_i}{p_j}  \right ) = 1
\qquad \text{for every } j \notin S
\end{equation}
and
\begin{equation} \label{eq:p1mod4constraint2}
\prod_{i \notin S} \left(  \dfrac{p_i}{p_j}  \right ) = 1
\qquad \text{for every } j \in S 
\end{equation}
(with the convention that $(p/2) = (2/p)$ for odd primes $p$ if $p_1 = 2$).

If $t$ is odd, then either the number of elements, $\vert S \vert$, in $S$, or the number of elements not in $S$, is odd.  
Since every quadratic residue symbol $(p_i/p_j)$ is $-1$ by assumption (again with the convention that
$(p/2) = (2/p)$ if $p_1 = 2$), this contradicts either equation \eqref{eq:p1mod4constraint1} if $\vert S \vert$ is odd
or equation \eqref{eq:p1mod4constraint2} if $\vert S \vert$ is even.
\end{proof}

\begin{remark}
Consider the rank 4 multiquadratic field $L = \QQ( \sqrt{p_1}, \sqrt{p_2}, \sqrt{p_3}, \sqrt{p_4})$  with 
distinct primes $p_1 \equiv p_2 \equiv p_3 \equiv p_4 \equiv 1$ mod 4 such that all quadratic residue symbols $(p_i/p_j)$, $1 \le i < j \le 4$
are equal to $-1$.  
Then all four simple quadratic subfields $\QQ(\sqrt{p_i})$, all six `double' quadratic subfields $\QQ(\sqrt{p_i p_j})$, and all four
`triple' quadratic subfields $\QQ(\sqrt{p_i p_j p_k})$ have fundamental units of norm $-1$ by Propositions \ref{prop:p}, \ref{prop:p1p2} and
\ref{prop:rnk3unitsminus1}. The remaining
quadratic subfield $\QQ(\sqrt{p_1 p_2 p_3 p_4})$ can have fundamental unit of norm $+1$ 
(computations suggest this is the more likely---and support the prediction,
obtained using a strong version of the heuristics of \cite{St}, that this possibility should occur 6/7 of the time)
as for the field $\QQ(\sqrt{ 5},\sqrt{ 13},\sqrt{37 },\sqrt{ 97} )$ for example,  or norm $-1$,
as for the field $\QQ(\sqrt{ 17},\sqrt{ 29},\sqrt{41 },\sqrt{ 97} )$.  In particular, the result in
Proposition \ref{prop:todd} does not hold in general for $t$ even, $t \ge 4$ (although it does hold
for $t = 2$ by Proposition \ref{prop:p1p2}).

\end{remark}

There are a number of interesting questions to investigate relating to 
the unit signature ranks in multiquadratic extensions, including those touched on above:
if $L$ is a real multiquadratic extension of rank $t$,
is it possible that every quadratic subfield of $L$ has a fundamental unit of norm $-1$?
if so, are there such fields $L$ where the maximum possible unit signature rank $2^t$ is attained? 
if all quadratic subfields have fundamental units
of norm $+1$ is it possible to construct a field $L$ with unit signature rank 1?

\begin{remark}
The results of \cite{DDK} show the signature rank of the units in the cyclotomic field of $m$th roots of unity tends to 
infinity with $m$, which might suggest that multiquadratic extensions with large rank $t$ having unit
signature rank 1 do not exist, but the possibility is certainly not ruled out. 
Further, if a (very) strong version of the heuristics in \cite{St} hold, namely: (a) that the proportions predicted in
\cite{St} hold for values of $d$ with a specified R\'edei matrix (up to conjugation by a permutation matrix), 
and (possibly less plausibly) (b) conditions in (a)
for values of $d$ constructed from different subsets of a single collection of primes can be
satisfied simultaneously and with independent probabilities, then the heuristics in \cite{St} would predict
the existence of infinitely many (although an extremely small proportion) of 
multiquadratic extensions of rank $t$ all of whose quadratic subfields have a fundamental unit of norm $-1$.
\end{remark}

\bigskip

We close by proving some results that show the unit signature rank deficiency of a real multiquadratic extension 
can be arbitrarily large.  Since the deficiency never decreases in a totally real extension of a totally real field,
this proves the unit signature rank deficiency can be arbitrarily large in real
cyclotomic fields (see Theorem \ref{thm:cyclotomicrank} for example), a result shown in \cite{DDK} conditional on the existence of infinitely
many cyclic cubic fields with a totally positive system of fundamental units.
(We note that a proof of the existence
of such cyclic cubic fields has recently been announced by Voight, et al, using techniques from
algebraic geometry.)

\begin{theorem} \label{theorem:qmultifields}
Suppose $q_1, q_2, \dots , q_{2t}$ are distinct primes with $q_1 \equiv q_2 \equiv \dots \equiv q_{2t} \equiv 3$ mod 4.
Then the field $L = \QQ( \sqrt{q_1 q_2}, \dots , \sqrt{q_{2t-1} q_{2t}})$ contains at least $t$ totally
positive units that are independent modulo squares in $L$, i.e., the unit 
signature rank deficiency $\delta(L)$ satisfies $\delta(L) \ge t$.
\end{theorem}

\begin{proof}
Let $\epsilon_i$ be the fundamental unit for the subfield $k_i = \QQ(\sqrt{ q_{2i-1} q_{2i} })$.  Then
by the discussion prior to Theorem \ref{thm:def1ranks2and3}, the integer $m_i$
associated to $\epsilon_i$ as in Proposition \ref{prop:m} 
equals $q_{2i-1}$ if $( \frac{q_{2i-1}}{q_{2i}} ) = +1$,
and equals $q_{2i}$ if $( \frac{q_{2i-1}}{q_{2i}} ) = -1$.

Suppose some product
\begin{equation*}
\epsilon_{1}^{a_{1}} \epsilon_{2}^{a_{2}} \cdots 
\epsilon_{t}^{a_{t}}  ,
\end{equation*}
where each exponent $a_{i}$ is either 0 or 1, is a square in $L$.  Since $ m_i$ and
$\epsilon_i$ differ by a square in $k_i$, it would follow that the integer
\begin{equation*}
m = m_{1}^{a_{1}} m_{2}^{a_{2}} \cdots 
m_{t}^{a_{t}}  
\end{equation*}
would be a square in $L$. 
Then $\QQ(\sqrt{m})$ would be a subfield of $L$, i.e., 
%
%
$m$ would differ by a rational square from some product
$(q_1 q_2)^{b_1} \dots (q_{2t-1} q_{2t})^{b_t}$ where the exponents $b_i$
are either 0 or 1.
Since the $q_i$ are distinct primes, it is clear that this can only
happen if $a_i = 0$ for every $i =1,2, \dots , t$

Hence $\epsilon_1 , \dots , \epsilon_t $ are totally positive units 
that are independent modulo squares in $L$, which proves the theorem.  
\end{proof}

\begin{theorem} \label{thm:cyclotomicrank}
Suppose the positive integer $n$ is divisible by at least $2t$ distinct primes congruent to 3 mod 4.  Then
the unit signature rank deficiency of the maximal real subfield $\QQ(\zeta_n)^+$ of the
cyclotomic field of $n$th roots of unity is at least $t$.  In particular, the 
unit signature rank deficiency for real cyclotomic fields can be arbitrarily large.
\end{theorem}

\begin{proof}
If $q_1, \dots , q_{2t}$ are distinct primes congruent to 3 mod 4 that divide $n$, then
$\QQ(\zeta_n)^+$ contains the subfield 
$\QQ ( \sqrt{q_1 q_2}, \dots , \sqrt{q_{2t - 1} q_{2 t}})$, which has deficiency at least
$t$ by Theorem \ref{theorem:qmultifields}.  Since 
the unit signature rank never decreases in a totally real extension of a totally real field,
the corollary follows.
\end{proof}

\begin{remark}
As mentioned in the Introduction, and used in the previous proof, 
if $F$ and $F'$ are totally real number fields with $F \subseteq F'$, then their
unit signature rank deficiencies satisfy $\delta(F') \ge \delta(F)$.  
We emphasize 
that this inequality is not, in general, due to totally positive units in
$F$ that are independent modulo squares in $F$ remaining independent
modulo squares in $F'$.  
For example, if $F = \QQ(\sqrt{q_1 q_2})$ and
$F' = \QQ(\sqrt{q_1}, \sqrt{q_2})$ with distinct primes $q_1 \equiv q_2 \equiv 3$ mod
4, then $\delta(F) \ge 1$, so $\delta(F') \ge 1$ (the precise possibilities
are given in Theorem \ref{thm:def1ranks2and3}), yet Theorem  \ref{thm:def1ranks2and3}
shows the fundamental unit in $F$ is always a square in $F'$.  If the integer
$n$ in Theorem \ref{thm:cyclotomicrank} is divisible by 4, then 
$\QQ(\zeta_n)^+$ contains all the fields $\QQ(\sqrt{q_{2i - 1}}, \sqrt{q_{2 i}})$,
so that {\it all\/} $t$ of the units used in the proof of Theorem \ref{thm:cyclotomicrank} 
to show that $\delta( \QQ(\zeta_n)^+ ) \ge t$ are
themselves squares in $\QQ(\zeta_n)^+ $, i.e., {\it none\/} of these units themselves contribute to the 
deficiency of $\QQ(\zeta_n)^+ $.

\end{remark}

The fields in Theorem \ref{theorem:qmultifields} show there are infinitely many 
multiquadratic extensions $L$ with
Galois group $(\ZZ / 2 \ZZ)^{t}$ containing at least $t$ units that are totally positive 
and independent modulo squares (and such a multiquadratic field $L$ having unit signature rank 1 would
require $L$ to have $2^t - 1$ such units).  
The following theorem shows there are multiquadratic fields for which we can prove
better lower bounds for their unit signature rank deficiency (but still
logarithmic in the degree of the field).  The multiquadratic fields are described explictly
in the proof---they are composites of the 
fields constructed in Theorem \ref{theorem:deficiency3} and 
Theorem \ref{thm:rank3totpos}.

\begin{theorem} \label{theorem:unboundeddeficiency}
Let $t \ge 1$ be any integer.  

\begin{enumerate}

\item
There exist infinitely many multiquadratic extensions $L$ with
Galois group $(\ZZ / 2 \ZZ)^{2 t}$ containing $3t$ units that are totally positive 
and independent modulo squares, i.e., $\delta(L) \ge 3t$.

\item
There exist infinitely many multiquadratic extensions $L$ with
Galois group $(\ZZ / 2 \ZZ)^{3 t}$ containing $7t$ units that are totally positive 
and independent modulo squares, i.e., $\delta(L) \ge 7t$.

\end{enumerate}
\end{theorem}

\begin{proof}
(1) Choose real biquadratic fields $K_1, K_2, \dots, K_t$ as in Theorem \ref{theorem:deficiency3},
each of which has unit signature rank deficiency 3 and whose odd discriminants are relatively
prime in pairs.  Let $L = K_1 K_2 \dots K_t$ be the composite field, so that $L$ is
a multiquadratic extension with Galois group $(\ZZ / 2 \ZZ)^{2 t}$.  
For each $i$ with 
$1 \le i \le t$, let 
$k_{i,1} = \QQ(\sqrt{d_{i,1}})$, 
$k_{i,2} = \QQ(\sqrt{d_{i,2}})$ and 
$k_{i,3} = \QQ(\sqrt{d_{i,3}})$ denote the three real quadratic
subfields of $K_i$ with corresponding totally positive fundamental units
$\epsilon_{i,1}$, $\epsilon_{i,2}$ and $\epsilon_{i,3}$,
which are independent modulo squares in $K_i$,
as in Theorem \ref{theorem:deficiency3}.

For $j = 1,2,3$, let $m_{i,j}$ denote the integer associated to $\epsilon_{i,j}$
in $k_{i,j}$ as in Proposition \ref{prop:m}, so that $m_{i,j}$ divides the odd integer $d_{i,j}$
and $m_{i,j} \epsilon_{i,j}$ is a square in $k_{i,j}$.

Suppose some product
\begin{equation*}
\epsilon_{1,1}^{a_{1,1}} \epsilon_{1,2}^{a_{1,2}} \epsilon_{1,3}^{a_{1,3}} \cdots 
\epsilon_{t,1}^{a_{t,1}} \epsilon_{t,2}^{a_{t,2}} \epsilon_{t,3}^{a_{t,3}} ,
\end{equation*}
where each $a_{i,j}$ is either 0 or 1, is a square in $L$.  
Then 
\begin{equation} \label{eq:mproduct}
m_{1,1}^{a_{1,1}} m_{1,2}^{a_{1,2}} m_{1,3}^{a_{1,3}} \cdots 
m_{t,1}^{a_{t,1}} m_{t,2}^{a_{t,2}} m_{t,3}^{a_{t,3}}
\end{equation}
would also be a square in $L$.  The product in \eqref{eq:mproduct} would therefore differ
by the square of a rational number from some 
product of the $d_{i,j}$ ($1 \le i \le t$ 
and $ 1 \le j \le 3$) since the square roots of such products generate $\QQ$ and
all the quadratic subfields of $L$.
By assumption, $d_{i,j}$ and $d_{i',j'}$ are relatively prime if $i \ne i'$ and so
$m_{i,j}$ and $m_{i',j'}$ are also relatively prime if $i \ne i'$.
It follows that, for each $i = 1, 2, \dots , t$, the product
$m_{i,1}^{a_{i,1}} m_{i,2}^{a_{i,2}} m_{i,3}^{a_{i,3}}$  in 
\eqref{eq:mproduct} would differ by the square of a rational number from some product of
$d_{i,1}$, $d_{i,2}$ and $d_{i,3}$.   But then
$\epsilon_{i,1}^{a_{i,1}} \epsilon_{i,2}^{a_{i,2}} \epsilon_{i,3}^{a_{i,3}}$
would be a square in $K_i$, which implies $a_{i,1} = a_{i,2} = a_{i,3} = 0$
since $\epsilon_{i,1}$, $\epsilon_{i,2}$ and $\epsilon_{i,3}$ are independent modulo squares
in $K_i$. 
Hence $a_{i,j} = 0$ 
for every $i$ and $j$, so the $3 t$ totally positive units $\epsilon_{i,j}$ for $1 \le i \le t$ and $1 \le j \le 3$ 
are independent modulo squares in $L$, which proves (1).

Applying the same proof, mutatis mutandis, using
Theorem \ref{thm:rank3totpos} instead of Theorem \ref{theorem:deficiency3} proves (2), so we
omit the details.
\end{proof}

\section*{Acknowledgments}
We would like to thank Evan Dummit and Richard Foote for helpful conversations.  We would
particularly like to thank Eduardo Friedman and Robert Auffarth, whose questions involving unit signature ranks
in cyclotomic extensions and their possible application to the number of principal 
polarizations of certain abelian varieties prompted us to take a closer look at multiquadratic extensions.

\end{document}